\documentclass[11pt,reqno]{amsproc}

\usepackage[margin=1.0in]{geometry}
\usepackage[english]{babel}
\usepackage[utf8]{inputenc}
\usepackage{ulem}
\usepackage{amsmath,amssymb,amsthm, comment, mathtools}
\usepackage{hyperref,mathrsfs,xcolor}

\usepackage{enumitem}

\usepackage{wrapfig}

\mathtoolsset{showonlyrefs}

\newcommand{\R}{\mathbb{R}}
\newcommand{\Ss}{\mathbb{S}}

\renewcommand{\L}{\mathcal{L}}

\newcommand{\pa}{\partial}

\newcommand{\rmd}{{\rm d}}

\newcommand{\be}{\begin{equation}}
\newcommand{\ee}{\end{equation}} 

\let\div\relax 

\DeclareMathOperator{\curl}{curl}
\DeclareMathOperator{\div}{div}

\numberwithin{equation}{section}
\newtheorem{theorem}{Theorem}
\numberwithin{theorem}{section}
\newtheorem{defi}{Definition}
\newtheorem{prop}{Proposition}
\numberwithin{prop}{section}

\numberwithin{cor}{section}
\newtheorem{lemma}{Lemma}
\numberwithin{lemma}{section}
\theoremstyle{remark}
\newtheorem*{remark}{Remark}

\title{On quasisymmetric plasma equilibria sustained by small force}
\author{Peter Constantin} \address{Department of Mathematics, Princeton University, Princeton, NJ 08544}\email{const@math.princeton.edu}
\author{Theodore D. Drivas} \address{Department of Mathematics, Princeton University, Princeton, NJ 08544}\email{tdrivas@math.princeton.edu}
\author{Daniel Ginsberg}\address{Program in Applied and Computational Mathematics, Princeton University, Princeton, NJ 08544}\email{ dg42@math.princeton.edu}

\begin{document}
\begin{abstract}
  We construct smooth, non-symmetric plasma equilibria which possess closed, nested flux surfaces and solve the Magnetohydrostatic (steady three-dimensional incompressible Euler) equations with a small force. The solutions are also `nearly' quasisymmetric. The primary idea is, given a desired quasisymmetry direction $\xi$, to change the smooth structure on space so that the vector field $\xi$ is Killing for the new metric and construct $\xi$--symmetric solutions of the Magnetohydrostatic equations on  that background by solving a generalized Grad-Shafranov equation.  If $\xi$ is close to a symmetry of Euclidean space, then these are solutions on flat space up to a small forcing.
\end{abstract}

\date{\today}
\maketitle

\section{Introduction}

Let $T\subset \R^3$ be a domain with smooth boundary. The three-dimensional Magnetohydrostatic (MHS) equations on $T$ read
\begin{alignat}{2}
J \times B &= \nabla P+ f, &&\quad \text{ in } T, \label{steadymhd}\\
 \nabla \cdot B &= 0, &&\quad \text{ in } T,\label{divB0}\\
 B\cdot \hat{n} &= 0, &&\quad \text{ on } \partial T,\label{Bbdy}
\end{alignat}
where $J= \nabla\times B$ is the current,  $f$ is an external force and $P$ is the
pressure. The solution $B$ to \eqref{steadymhd}--\eqref{Bbdy} can be interpreted as either a stationary fluid velocity field which solves the time-independent Euler equation, or as a steady self-supporting magnetic field in a continuous medium with trivial flow velocity.  The latter interpretation is robust across a variety of magnetohydrodynamic models (e.g. compressible, incompressible, non-ideal) and makes the system \eqref{steadymhd}--\eqref{Bbdy}
central to the study of plasma confinement fusion.

In view of this, there is a long standing scientific program to identify and construct magnetohydrostatic equilibria which are effective at confining ions during a nuclear fusion reaction. The most basic requirement for confinement is the existence of
a ``flux function'' $\psi$, whose level sets foliate the domain $T$ and which
satisfies $B\cdot \nabla \psi = 0$. To first approximation, ions move
along the integral curves of $B$ and so this condition ensures that particle
trajectories are
approximately constrained to the level sets of $\psi$.  For this reason, it is desirable to seek equilibria with \textit{nested flux surfaces} (isosurfaces of $\psi$) which foliate the plasma domain.
When $T$ is the axisymmetric torus,
it is natural to look for such solutions in the form of axisymmetric
magnetic fields.  If $(R, \Phi, Z)$ denote the usual cylindrical coordinates on $\R^3$ and the
center line of the
torus lies in the $Z = 0$ plane, axisymmetric solutions take the form
\begin{equation}
 B_0 = \frac{1}{R^2} \bigg(C_0(\psi_0) Re_\Phi + Re_\Phi \times \nabla \psi_0\bigg),
 \label{axiB}
\end{equation}
with flux function $\psi_0$.
In order for $B_0$ to satisfy \eqref{steadymhd} with $P_0 = P_0(\psi_0)$, taking $f = 0$ momentarily for simplicity and taking
$T$ to be the torus with inner radius $R_0 -1$ and outer radius $R_0 + 1$,
say,
the flux function needs to satisfy
the axisymmetric Grad-Shafranov equation (\cite{G67,S66})
\begin{alignat}{2}
  \pa_R^2 \psi_0 + \pa_Z^2\psi_0 - \frac{1}{R}\pa_R \psi_0
  + R^2 P_0'(\psi_0) + C_0C_0'(\psi_0) &= 0, && \qquad \text{ in } D_0,
 \label{axigs}
 \\
 \psi_0 &= \textrm{const. }
 && \qquad \text{ on } \pa D_0,
 \label{axibc}
\end{alignat}
where $D_0$ denotes the cross-section of the torus (unit disk) in the $\Phi = 0$
half-plane centered at $R = R_0$.
Conversely, if $\psi_0$ is any solution\footnote{We remark that it could be that this equation admits ``large" solutions with non-trivial dependence on $\Phi$.  See the work of Garabedian  \cite{G06}.} to \eqref{axigs} with
$e_\Phi\cdot\nabla \psi_0 = 0$ then the vector field $B_0$ defined in \eqref{axiB}
is divergence-free and satisfies \eqref{steadymhd}. If $\psi_0$ is constant on $\pa T$, $B_0$ satisfies \eqref{Bbdy}.

Unfortunately, these tokamak equilibria come with a slew of problems from the point of view of plasma confinement fusion \cite{L19}.  For example, to achieve
improved confinement it is desirable for the magnetic field to `twist' as
it wraps around the torus and this can only be accomplished in axisymmetry
with a large plasma current, $J$. Such plasma configurations are hard to control
in practice.
One approach to finding equilibria with better confinement properties is to
consider equilibria in geometries which have the desired twist built in.
This is the basic design principle behind the stellarator, \cite{GB91}. It is still desirable
for these configurations to possess a form of symmetry, which is known
as quasisymmetry.

\begin{defi} [Weak quasisymmetry, \cite{RHB20}]  \label{qsdef}
  Let $\xi$ be a non-vanishing vector field tangent to $\partial T$.
  We say that $\xi$ is a \textit{quasisymmetry}  and the
  field $B$ is \textit{quasisymmetric} with respect to $\xi$ if
  \begin{alignat}{2}
     \div \xi & =0,&&\quad \text{ in } T,\label{qs1intro}\\
     B \times \xi&=-\nabla \psi,&&\quad \text{ in } T,\label{qs2intro}\\
  \xi \cdot  \nabla |B|&=0, && \quad \text{ in } T,
     \label{qs3intro}
\end{alignat}
for some function $\psi: T\to \mathbb{R}$.
\end{defi}
The significance of the condition \eqref{qs2intro} is
that it implies $B\cdot \nabla \psi = 0$
and $\xi \cdot \nabla \psi = 0$ and so quasisymmetric solutions posses flux
functions which are symmetric with respect to $\xi$. In \cite{RHB20}, the authors
 argue that \eqref{qs1intro}--\eqref{qs3intro} form sufficient conditions
 that ensure first order (in gyroradius) particle confinement, hence the terminology of \textit{weak} quasisymmetry.
 In the confinement fusion literature \cite{L19,BKM19}, one encounters the following alternative
definition which is actually stronger than the above.  It replaces \eqref{qs3intro} with
\be\label{sqs3intro}
\xi \times J= \nabla (B \cdot \xi)\quad \textit{ in } T.
\ee
We term this set of conditions \textit{strong} quasisymmetry.
When $f=0$ it is this stronger form of quasisymmetry which is equivalent to other definitions in the plasma fusion literature involving Boozer angles, see \S 8 of \cite{L19}.
If $\div B = 0$ then \eqref{qs3intro} requires only that  a single component of \eqref{sqs3intro} vanish, $B\cdot \big( \xi \times J
- \nabla (B\cdot \xi)\big) = 0$.\footnote{To see this, using standard vector calculus identities, we write
\be
\xi \times \curl B - \nabla (B\cdot \xi) = B\cdot \nabla \xi + \xi \cdot \nabla B + B\times \curl \xi.
\ee
Taking the inner product with $B$  results in $B \cdot (\xi \times \curl B - \nabla (B\cdot \xi) )= \frac{1}{2} \xi \cdot \nabla |B|^2+ B\cdot \nabla \xi \cdot B$.
The argument is completed by using the elementary
identity $\L_\xi B = \curl(B\times \xi) + \div B \xi - \div \xi B
= \curl (B\times \xi)$.
This yields $B \cdot (\xi \times \curl B - \nabla (B\cdot \xi) )=  \xi \cdot \nabla |B|^2$.
}
In light of this, the additional content of strong quasisymmetry \eqref{qs1intro}, \eqref{qs2intro}, and \eqref{sqs3intro} is the assumption  that the other two components of $\xi \times J
- \nabla (B\cdot \xi)$ vanish.  It turns out that when there is no force and the equilibria are toroidal, strong quasisymmetry is equivalent to Definition \ref{qsdef}.\footnote{M. Landreman, private communication.}

From \eqref{qs2intro}, if $\xi\cdot B$ is constant
on surfaces of constant $\psi$, $\xi \cdot B = C(\psi)$ (by a result in \cite{BKM19},
any solution of \eqref{steadymhd} with $f = 0$ which satisfies \eqref{qs1intro}-\eqref{qs3intro} satisfies this condition),
it follows that $B$ is of the form
\begin{equation}
 B = \frac{1}{|\xi|^2} \bigg(C(\psi) \xi + \xi \times \nabla \psi\bigg),
 \label{Bqs}
\end{equation}
and when $f = 0$, the requirement that
\eqref{steadymhd} holds implies that
$\psi$ must satisfy the quasisymmetric Grad-Shafranov
equation (introduced in \cite{BKM19}) which reads
\begin{align}
\Delta \psi - \frac{\xi \times \curl \xi}{|\xi|^2} \cdot \nabla \psi
 + \frac{\xi \cdot  \curl \xi}{|\xi|^2}  C(\psi) +  CC'(\psi)  + |\xi|^2P'(\psi)&=0, \qquad \qquad \text{in}  \ T,  \label{qgs1}\\
 \psi &= {\rm const. }  \ \ \ \ \ \  \ \text{on}  \ \partial T.
 \label{qgs}
\end{align}

The equations
\eqref{divB0} and \eqref{qs1intro}, \eqref{qs3intro} can be thought of
as constraints relating $\psi$ to the deformation
tensor of $\xi$, the symmetric two-tensor $\L_\xi \delta$ defined by
\begin{equation}
 (\L_\xi\delta)(X, Y) = \nabla_X \xi \cdot Y + \nabla_Y \xi \cdot X,
 \label{}
\end{equation}
where $\nabla$ denotes covariant differentiation with respect to the Euclidean
metric. Recall that $\xi$ generates an isometry of Euclidean space
if and only if $\L_\xi \delta= 0$,
in which case $\xi$ is called a Killing field for
the metric $\delta$.  Assuming
that $\xi\cdot \nabla \psi = 0$ and $\div \xi = 0$, from \eqref{Bqs} we find
\begin{equation}
 \div B = C(\psi) (\L_\xi\delta)(\xi, \xi) + (\L_\xi \delta)(\xi, \xi\times
 \nabla \psi),
 \label{divBxi}
\end{equation}
and expanding the condition \eqref{qs3intro} we find
\begin{equation}
 \frac{1}{2}\L_\xi |B|^2 =
 (\L_\xi\delta)(\xi, \xi) + \frac{2}{C(\psi)}
  (\L_\xi\delta)(\xi, \xi \times \nabla \psi)
   + \frac{1}{C(\psi)^2} (\L_\xi\delta)(\xi\times \nabla \psi,
 \xi \times \nabla \psi),
 \label{qsdefo}
\end{equation}
see Lemma \ref{propBLem} of Appendix \ref{structural}.
The equation \eqref{qsdefo} is a complicated relationship between
$\psi, C(\psi)$ and $\xi$ but notice that it holds trivially
(assuming only that $\xi\cdot \nabla \psi = 0$) whenever
when $\xi$ is a Killing field.
It is well-known that in Euclidean space the only Killing fields are
linear combinations of translations and rotations. Therefore, up to
a multiplicative constant, the only such field compatible with the geometry of the axisymmetric
torus is $\xi = R e_\Phi$ and as mentioned above, such solutions
have problematic confinement properties. We have arrived at the following problem.
\vspace{2mm}

\noindent
\textbf{Problem:}\textit{
Given a toroidal domain $T$, construct a function $\psi:T\to \R$ with nested flux surfaces and a
divergence-free
vector field $\xi$ which does not
generate an isometry of $\R^3$ and is tangent to $\pa T$,
so that
\eqref{qgs1}, \eqref{qs3intro}, the nonlinear constraints \eqref{divBxi}, \eqref{qsdefo} and
 $\xi\cdot \nabla \psi = 0$ all hold.}
\vspace{2mm}

It is not clear that there are any smooth solutions
$\psi, \xi$ to the above problem. In fact, in
1967 (long before the above notion of quasisymmetry was introduced),
Grad \cite{GR58,G67,G85} conjectured that the only smooth solutions
to \eqref{steadymhd}--\eqref{Bbdy} possessing a good flux function have a Euclidean symmetry
\footnote{Specifically, in \cite{G67} Grad conjectures that there no \textit{families} of smooth
solutions to \eqref{steadymhd}-\eqref{Bbdy}, each posessing a flux function with closed
level sets that foliate the domain $T$, other than the axisymmetric solutions. This leaves open the
possibility of isolated non-axisymmetric steady states, far from symmetry.},
and this would in particular rule out any solutions of the above type.
 Since Grad's work, there have been some constructions of non-symmetric equilibria an infinite cylindrical domains \cite{SK95,SK97}.  As these are unbounded in extent, they have limit practical appeal for the perspective  of confinement. No such examples of smooth solutions have been rigorously  demonstrated on toroidal domains, although there has been some work on suggestive  formal near-axis expansions \cite{W14,BMT86,JSL19} and non-symmetric weak solution equilibria with pressure jumps have been rigorously constructed \cite{Bl69} which may have practical implication for the confinement fusion program \cite{H1,H2}.\footnote{See Lortz \cite{L70} for a construction of a non-axisymmetric toroidal equilibrium which nevertheless enjoys plane reflection symmetry (forcing all magnetic field lines to be closed).}

We do not address Grad's conjecture here and our goal is instead to present a robust
method for constructing solutions to \eqref{steadymhd} with small force and which
are approximately quasisymmetric with respect to a given vector field $\xi$
(sufficiently close to the axisymmetric vector field $\xi_0 = R e_\Phi$),
in the sense that \eqref{qs2intro} holds but
that \eqref{qs3intro} holds up to a small error.

In addition to the nontrivial constraint \eqref{qs3intro}, there are two
serious difficulties in constructing solutions to \eqref{steadymhd}--\eqref{Bbdy}
of the form \eqref{Bqs}
with given symmetry direction $\xi$. The first is that by \eqref{divBxi}, unlike
in the axisymmetric setting, vector fields of the form \eqref{Bqs}
need not be divergence-free. The second difficulty is that
for arbitrary $\xi$, it is not at all clear that
the equation \eqref{qgs1}--\eqref{qgs} admits any solutions with
$\xi\cdot \nabla \psi = 0$, since the coefficients appearing in
\eqref{qgs1}--\eqref{qgs} need not be invariant under $\xi$. Both of these difficulties
can be traced to the fact that $\xi$ need not be a Killing field with
respect to the Euclidean metric. To circumvent these issues, inspired by \cite{LMP19} and \cite{Burby},
we replace the metric structure of $(\R^3, \delta)$ with $(\R^3, g)$
for a metric $g$ for which $\xi$ is a Killing field. The resulting
magnetic field will not satisfy the usual MHS equations \eqref{steadymhd}
but provided $\xi$ is sufficiently close to Killing for the Euclidean
metric, the error will be small.
We now explain the idea.

Let us suppose that given $\xi$, we can find a metric $g$ on $\R^3$
for which $\L_\xi g= 0$, that is, for which $\xi$ generates an isometry
(we give an explicit construction of such metrics for a large class
of vector fields $\xi$ after the upcoming statement of
Theorem \ref{noqsthmvague}).
We then consider the following generalization of the
ansatz \eqref{Bqs}, introduced in \cite{Burby}
\begin{equation}
 B_g = \frac{1}{|\xi|_g^2} \bigg(
 C(\psi)\xi +  \sqrt{|g|} \xi\times_g \nabla_g \psi\bigg).
 \label{gBqs}
\end{equation}
Here, $|\xi|_g, \times_g, \nabla_g$ denote the analogs of the usual
Euclidean quantities $|\xi|,\times \nabla$ with respect to the metric
$g$ (see Appendix \ref{notation}).
In Lemma \ref{propBlem1} we use the fact that $\L_\xi g = 0$
to show that vector fields of this form
are divergence free assuming only that $\xi\cdot \nabla \psi = 0$,
\begin{equation}
 \div B_g = 0,
 \label{divBg}
\end{equation}
and also that $\psi$ is a flux function for $B_g$,
\begin{equation}
  \xi \times \nabla \psi = B_g.
 \label{xiflux}
\end{equation}
We emphasize the somewhat surprising fact that even though the definition
\eqref{gBqs} involves the metric $g$ in a nontrivial way, it is designed that way so that the identities
\eqref{divBg}-\eqref{xiflux} involve only Euclidean quantities.  We remark that $B_g$ will \text{not} be divergence-free with respect to the $g$ metric.

We then seek $B_g$ of the form \eqref{gBqs}
which satisfy the MHS with respect to the metric $g$,
\begin{equation}
 \curl_g B_g \times_g B_g = \nabla_g P.
 \label{caMHD}
\end{equation}
This ansatz leads to the generalized Grad-Shafranov equation for $\psi$
\begin{align}
  \div_g\bigg(\sqrt{|g|} \frac{\nabla_g \psi}{|\xi|_g^2} \bigg) - C(\psi)
 \frac{\xi}{|\xi|_g^2} \cdot_g \curl_g
 \left(\frac{\xi}{|\xi|_g^2}\right)
+ \frac{C(\psi)C'(\psi)}{\sqrt{|g|}|\xi|_g^2}+
\frac{P'(\psi)}{\sqrt{|g|}}
&= 0, \qquad \qquad \text{in}  \ T,  \label{gGS}\\
 \psi &= {\rm (const.)}  ,  \ \ \  \text{on}  \ \partial T. \label{gGSbc}
\end{align}
where $|\xi|_g, \cdot_g, \curl_g$ denote the magnitude, dot product and curl
with respect to the metric $g$ (see Appendix \ref{notation} for the definitions
and Appendix \ref{appendproof} for the derivation of \eqref{gGS} from
\eqref{gBqs} and \eqref{caMHD}).  Note that \eqref{gGS}--\eqref{gGSbc}  reduces to \eqref{qgs1}--\eqref{qgs} when $g=\delta$, and when $g$ is the circle-averaged metric, it agrees with
the equation derived in \cite{Burby}. As shown in \cite{Burby}, all solutions
of MHS \eqref{steadymhd}-\eqref{Bbdy} without force and non-vanishing pressure
gradient must have a flux function satisfying \eqref{gGS} where $g$
is the circle-averaged metric discussed below. In light of this, the study
of the generalized Grad-Shafranov equation \eqref{gGS} is of fundamental importance
in the study of solutions to MHS with a generalized symmetry.

As another consequence of the fact that $\L_\xi g= 0$,
the coefficients in the equation \eqref{gGS}
are invariant under $\xi$ and so \eqref{gGS}, unlike \eqref{qgs}, is consistent with the requirement
$\xi\cdot \nabla \psi = 0$.
The downside is that the equation \eqref{caMHD} does not agree with
\eqref{steadymhd} unless $g = \delta$ and so $B_g$ will
not satisfy the original MHS equations. However,
if we can arrange for the metric $g$
to be sufficiently close to the Euclidean
metric $\delta$, then $B_g$ will satisfy the usual MHS equations
 $\curl B_g \times B_g - \nabla P = 0$ up to a small
 error.
Our approach will be to solve the generalized Grad-Shafranov
equation \eqref{gGS} by deforming an appropriate solution $\psi_0$
of the axisymmetric Grad-Shafranov equation \eqref{axigs},
using the methods from \cite{CDG20}. 	In particular, we seek a diffeomorphism $\gamma: D_0 \to D$ and requiring that $\psi = \psi_0\circ \gamma^{-1}$. It turns out (see section \ref{mainthmpf}) that this reduces to a system of nonlinear elliptic equations for the components of $\gamma$ which can be solved by a iteration.

In what follows, $T_0$ denotes the axisymmetric torus
\begin{equation}
 T_0 = \{ (R,\Phi, Z)\ |\ (R - R_0)^2 + Z^2 \leq a, 0\leq \Phi \leq 2\pi\},
 \label{}
\end{equation}
with thickness $0 < a \ll R_0$. Let $\xi_0 = R e_\Phi$ be the generator of rotations in the $Z = 0$ plane.
Let $D$ be any domain in
the half-plane $\{\Phi = 0\}$ sufficiently close to $D_0$.
Suppose that $\xi$ is a vector field which is
sufficiently close to the rotation field $\xi_0$
with the property that all the orbits of $\xi$ starting
from $D$ are periodic (with possibly different period
$\tau(p)$).
In this case we define the toroidal domain
\begin{equation}
 T =
 \{\varphi_s(p) \ | \  p\in D, \ \ s\in [0, \tau(p))\},
 \label{torusdef}
\end{equation}
where $\varphi_s(p)$ denotes the time-$s$ flow of $\xi$
starting from $p \in D$,
\begin{equation}
 \frac{\rmd }{\rmd s} {\varphi}_s(p) = \xi(\varphi_s(p)), \qquad
 \varphi_0(p) = p \in D.
 \label{}
\end{equation}
In this setting we say that the toroidal domain $T$ is \textit{swept out
by $\xi$ from D}.

\begin{figure}[h!]
  \includegraphics[height=2.1in]{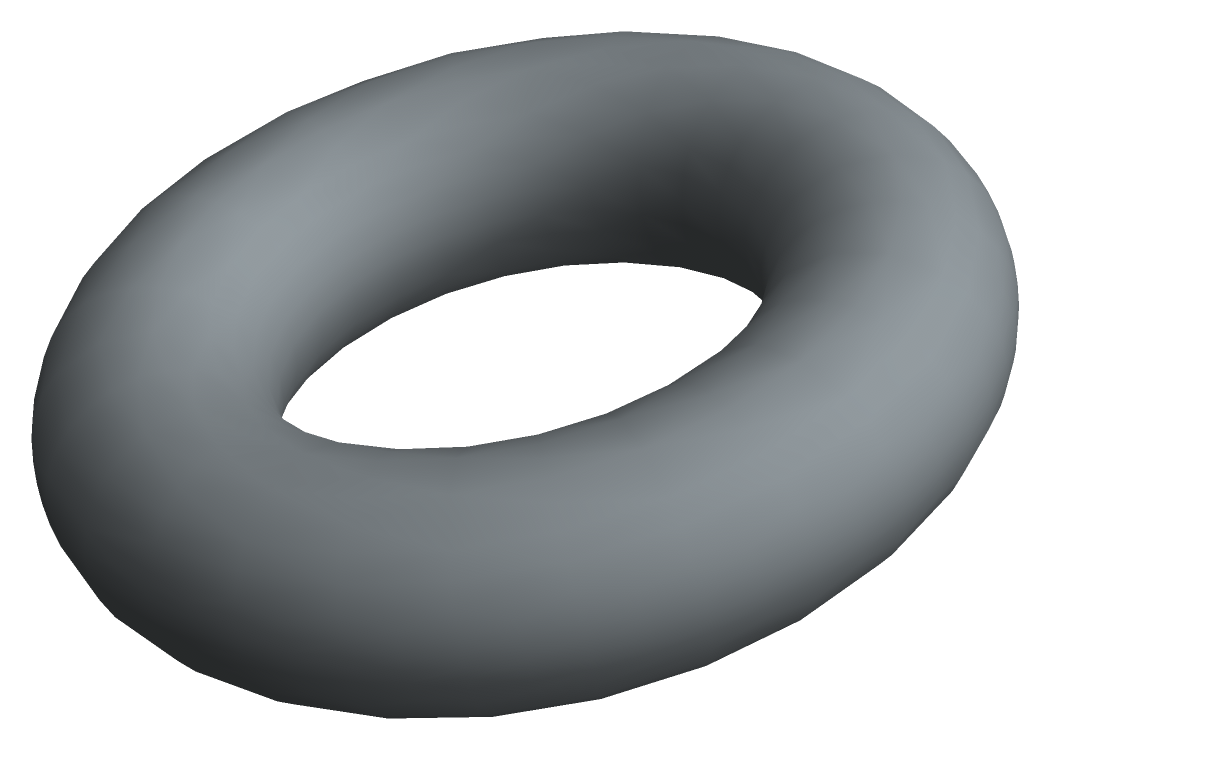}
  \includegraphics[height=2.1in]{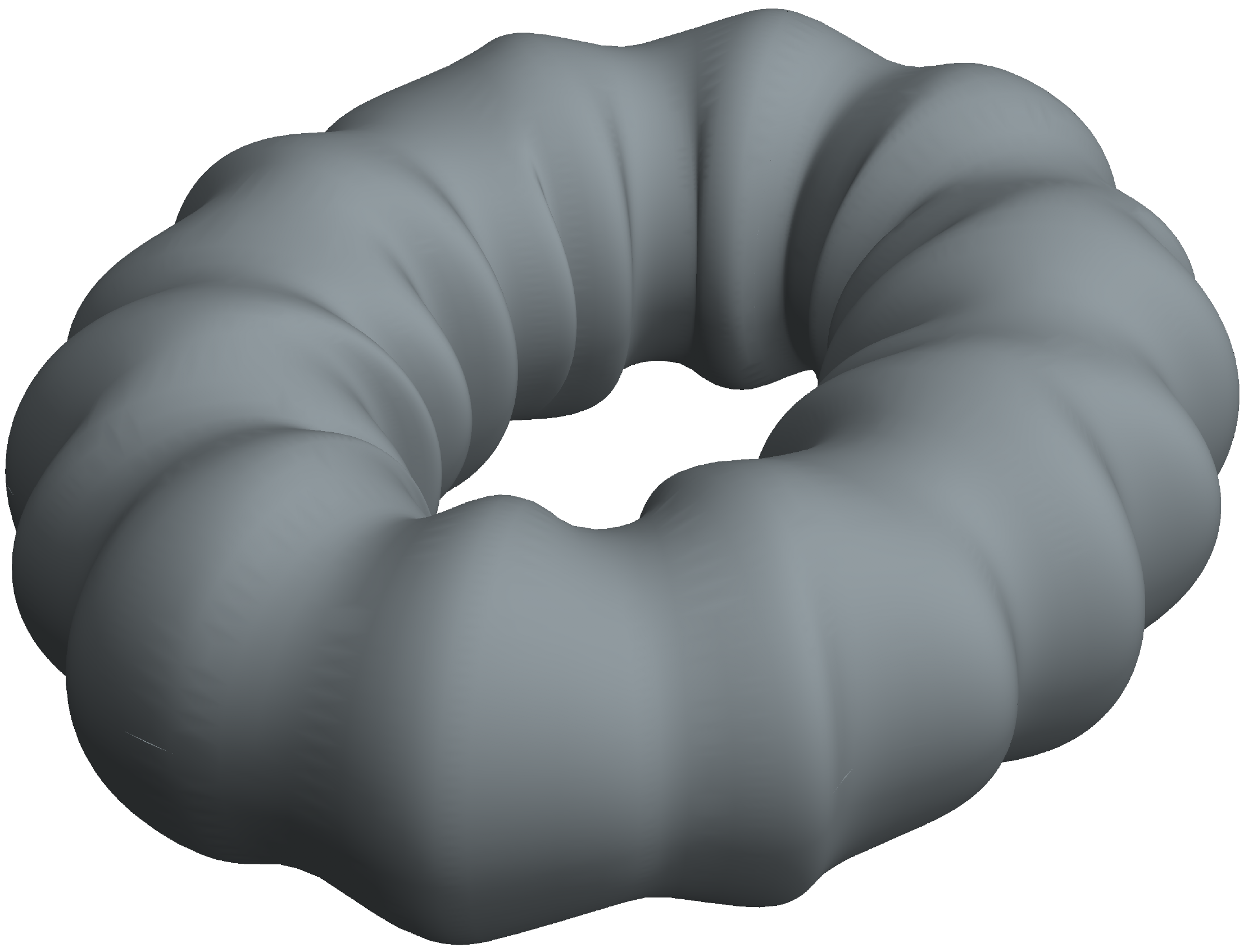}
  \caption{Tokamak and stellarator geometries.}
  \label{fig:stell}
\end{figure}

Our first result is that, given a toroidal domain $T$ swept out by
a vector field $\xi$ as above,
sufficiently close to the axisymmetric torus $T_0$, we can find a flux function
satisfying the generalized Grad-Shafranov equation \eqref{gGS}.
The proof is constructive and relies on deforming a known axisymmetric steady state satisfying mild conditions {\rm (\textbf{H1})}--{\rm (\textbf{H2})} stated in
 \S \ref{mainthmpf}.

\begin{theorem}
  \label{noqsthmGS}
  Fix $k \geq 0, \alpha > 0$  and let $\xi \in C^{k+2,\alpha}(\R^3)$ be a divergence-free vector field,
  sufficiently close in $C^{k+2,\alpha}$ to the
  rotation vector field $\xi_0 = R e_\Phi$.
  Let $D$ be a domain sufficiently close
  to $D_0$ in $C^{k+2,\alpha}$
  in the sense that $D =  \{(r, \theta) \  |\  0 \leq r \leq \mathfrak{b}(\theta), \theta
  \in \Ss^1\}$ for a function $\mathfrak{b}:\Ss^1 \to \R$ sufficiently close to 1
  in $C^{k+2,\alpha}(\Ss^1)$. Let $\psi_0 \in C^{k+2,\alpha}(D_0)$
  be a solution of \eqref{axigs}-\eqref{axibc} with pressure
  $P_0 \in C^{k+1,\alpha}(\R)$ and with $C_0 \in C^{k+1,\alpha}(\R)$
  satisfying {\rm (\textbf{H1})}--{\rm (\textbf{H2})}.

  Suppose moreover that
   $\xi$ has closed integral curves that sweep out a toroidal domain
  $T$ from $D$. Suppose
  that there is a metric $g\in C^{k+2,\alpha}(\mathbb{R}^3)$ with the property $\L_\xi g= 0$
  which is sufficiently close to the Euclidean metric.
  Then, for any given $C\in C^{k+1,\alpha}(\R)$ sufficiently close to $C_0$, there is a flux
 function $\psi \in C^{k+2,\alpha}(T)$, and a pressure
 $P = P(\psi) \in C^{k+1,\alpha}(T)$ so that $\psi$ satisfies the
 generalized Grad-Shafranov equation \eqref{gGS} and the boundary condition
 \eqref{gGSbc}. Moreover, the level sets of $\psi$ are diffeomorphic to
 the level sets of $\psi_0$.
\end{theorem}

As a consequence of the above theorem, we are able to produce magnetic fields with nested flux surfaces and a global symmetry that solve MHS up to a small force whose magnitude is controlled by the deviation of the symmetry from being Euclidean.  These fields satisfy two of the three quasisymmetry conditions, the third holding approximately. The resulting magnetic field possesses flux surfaces which have the same topology as the axisymmetric base state.
\begin{theorem}
  \label{noqsthmvague}
  Suppose the hypotheses of the previous theorem hold.
  For any given $C\in C^{k+1,\alpha}(\R)$ sufficiently close to $C_0$, there is a flux
 function $\psi \in C^{k+2,\alpha}(T)$, and a pressure
 $P = P(\psi) \in C^{k+1,\alpha}(T)$ so that the magnetic field $B$ defined by \eqref{gBqs} satisfies
 $B\cdot \nabla\psi = 0$, $B\times \xi  = \nabla \psi$ as well as  MHS \eqref{steadymhd}--\eqref{Bbdy} with a force $f$ obeying
 \begin{equation}
  \|f\|_{C^{k,\alpha}(T)} \leq c \|\delta - g\|_{C^{k+2,\alpha}(T)},
  \label{mhdwforcebd}
 \end{equation}
 where $c:= c(\|\xi\|_{C^{k+2,\alpha}(T)},\| \psi\|_{C^{k+2,\alpha}(T)},\| P_0\|_{C^{k+1,\alpha}(\mathbb{R})},\| C_0\|_{C^{k+1,\alpha}(\mathbb{R})})$.
 Moreover, the  flux surfaces of $B$ (isosurfaces of $\psi$) are diffeomorphic to the  isosurfaces of $\psi_0$,
 and $\psi$ is a solution of the generalized Grad-Shafranov equation \eqref{gGS}.
\end{theorem}

The point of the bound \eqref{mhdwforcebd} is that
if $\xi$ is a Killing field for the Euclidean metric $\delta$
then we can take $g = \delta$ in the above and by \eqref{mhdwforcebd},
the $B$ is then an exact solution of the MHS equations
\eqref{steadymhd} with $f = 0$. In this sense, \eqref{mhdwforcebd} shows that one
can construct approximate solutions to MHS with symmetry direction $\xi$
with error proportional to how far $\xi$ is from being a symmetry of
$\R^3$. We remark that the proof is quantitative in that all the small parameters can be explicitly defined in terms of the inputs $\psi_0$, $C_0$, $P_0$, and $\xi$.
Let us also remark that one is not free to choose $P$ from the outset
and it is instead determined in the course of the proof to enforce
a certain compatibility condition, see section \ref{mainthmpf}.

The above theorem is perturbative, in the sense that the resulting magnetic
field
will be approximately axisymmetric and have a flux function close to a
given $\psi_0$ satisfying the axisymmetric Grad-Shafranov equation \eqref{axigs}.
As will be discussed in the upcoming section, the result follows from a theorem
in \cite{CDG20} by deforming
the given solution $\psi_0$ of the axisymmetric Grad-Shafranov equation
\eqref{axigs} into a solution of the generalized
Grad-Shafranov equation \eqref{gGS}. {The same theorem from
\cite{CDG20} in fact allows one to deform a given solution $\psi_1$ to
\eqref{gGS} for given $\xi = \xi_1$ into a solution $\psi_2$ to \eqref{gGS} with nearby,
but different, $\xi = \xi_2$. Given a desired $\xi$, if one can produce a
sequence of vector fields $\xi_0, \xi_1,..., \xi_{N-1}, \xi_N = \xi$ in such a way
that the resulting solutions $\psi_0, \psi_1,..., \psi_{N-1}$ all satisfy the
conditions \textbf{(H1)}--\textbf{(H2)} this would produce a flux function
satisfying \eqref{gGS} far from axisymmetry. Note however that the resulting
force in \eqref{steadymhd} could be quite large.}

{
If one is only interested in constructing approximate equilibria, this
can be achieved simply by pushing forward a given axisymmetric state by a volume
preserving diffeomorphishm. The resulting flux function need not satisfy
the generalized Grad-Shafranov equation \eqref{gGS}.
On the other hand, the construction in Theorem \ref{noqsthmGS} does ensure that the
generalized Grad-Shafranov equation is exactly satisfied. In light of the fact that
\cite{Burby} shows that all unforced
solutions to \eqref{steadymhd}-\eqref{Bbdy} must satisfy the equation \eqref{gGS}, our theorem
may provide a path towards obtaining non-axisymmetric solutions without force.}

We now describe how to produce a  base state $\psi_0$ and metric $g$ which are suitable inputs for Theorem \ref{noqsthmGS}.  In Appendix  \S \ref{examplesec}, we provide an example of a base state $\psi_0$ satisfying {\rm (\textbf{H1})}--{\rm (\textbf{H2})} living on a large aspect-ratio torus. This is obtained as a perturbation of an explicit profile on an ``infinite aspect-ratio" torus. It should be stressed that the conditions {\rm (\textbf{H1})}--{\rm (\textbf{H2})} are not very stringent and should hold for a wide class of axisymmetric solutions that possess simple nested flux surfaces (which could e.g. be numerically obtained).  Next, we describe two large classes of vector fields $\xi$
and metrics $g$ satisfying the hypotheses of our theorem.
\begin{remark}[Designer metrics]
We provide two possible ways of constructing a `near' Euclidean metric given a `near' isometry $\xi$.
\begin{enumerate}
\item  \textbf{(Deformed metric)}: Suppose that the torus $T$
  is given by $T = f(T_0)$ where $f$ is a diffeomorphism
  defined in a neighborhood of $T_0$ and which is sufficiently
  close to the identity. Then we can take
  $\xi = df(\xi_0)$ where $df$ denotes the differential
  and let $g = f^*\delta$ denote the pullback of the Euclidean
  metric $\delta$ by $f$. Because the Lie derivative is
  invariant under diffeomorphisms, we have   $\xi$ is a Killing field for $g$ since
$\L_\xi g = f^* (L_{\xi_0} \delta) = 0$.
  \item  \textbf{(Circle averaged metric)}: Suppose the orbits of
$\xi$ starting from $D$ are all $2\pi$-periodic.
In this case we say that $\xi$ generates a circle-action.
Defining the circle-averaged metric $g$,
\begin{equation}
 g = \frac{1}{2\pi}\int_{0}^{2\pi} \varphi_s^* \delta \, \rmd s,
 \label{}
\end{equation}
it follows by a simple computation that $\L_\xi g = 0$.  Moreover, when $\xi$ is a Killing field
for Euclidean space, $g = \delta$.  This metric was introduced
by \cite{Burby}. As motivation for the appearance of this particular metric,  consider the MHS in terms of one-forms $\L_B(B^\flat) = d P$ (see \cite{AK99}).  In this representation, it is clear that the metric appears linearly (in the definition of $\flat$).  Therefore, if $B$ and $P$ are invariant under the flow of $\xi$, then one finds $\L_B(B^{\overline{\flat}}) = d P$ where $\overline{\flat}$ denotes lower the index with the circle average metric.  Raising indices with $g$, we find that any such MHS solution on Euclidean space is also a solution of the circle averaged equation (MHS with respect to the metric $g$).

\end{enumerate}

\end{remark}

We conclude with some remarks about achieving exact quasisymmetry.  By construction, the magnetic field $B$ from the previous theorem will
satisfy \eqref{qs2intro} but will only approximately satisfy the property \eqref{qs3intro} of quasisymmetry.  Thus, our fields confine particles to zeroth but not first order in the guiding center approximation \cite{RHB20}. The error from being an exact weak quasisymmetry can be easily quantified; for a vector field $B_g$ of the form \eqref{gBqs},
assuming that $g$ is such that $\L_\xi g = 0$
the condition \eqref{qs3intro} reads
\be\label{qsform}
\xi \cdot\nabla |B_g| = \frac{C^2(\psi)}{2|\xi|_g^4|B_g|} \bigg[ (\L_\xi \delta) (\xi ,\xi)  +2 C^{-1}(\psi)  (\L_\xi \delta) (\xi ,\nabla_g^\perp \psi ) +C^{-2}(\psi) (\L_\xi \delta) (\nabla_g^\perp \psi ,\nabla_g^\perp \psi ) \bigg].
\ee
See Lemma \ref{propBLem}. Since \eqref{qsform} involves the Euclidean deformation tensor alone, it is controlled by the deviation of $\xi$ from being a Euclidean isometry and our solution will have $\xi \cdot\nabla |B_g|$ small. The error from being a strong quasisymmetry is also quantifiably small.

It is worth remarking that there are additional freedoms in our construction that could, in principle, be used to further constrain the constructed solution.  Specifically, in our theorem, we treat $\xi$ as a fixed vector field sufficiently
 close to $\xi_0$ and we made the somewhat arbitrary choice that the
 map $\gamma$ should be volume preserving. The results in \cite{CDG20}
 actually allow one to construct the map $\gamma$ so that $\det\nabla \gamma
 := \rho$ is any given function, sufficiently close to one; in fact
 by iterating that result, one can additionally achieve that  $\det \nabla \gamma
 = X(\phi, \eta, \pa\phi, \pa\eta,\partial \pa_s\phi,\partial \pa_s\eta)$
 for a suitable nonlinearity $X$ sufficiently close to one when $\phi, \eta = 0$.
{Using this freedom, it is possible to show that, under some (possibly restrictive and undesirable) assumptions on the field $\xi$,
the Jacobian $\rho$ can be used to achieve exact quasisymmetry on a slice of the torus (namely on the cross-section $D$). Ensuring this property holds seems of little practical interest for ion confinement in a stellarator, since particles starting on the slice
will immediately leave. In contrast, Garren and Boozer \cite{GB}
and Plunk and Helander \cite{PH}
show that exact quasisymmetry is possible to achieve on one flux surface while
maintaining the MHS force balance, which is of greater relevance to
confinement in a stellarator. It is unknown whether or not
quasisymmetry can be achieved in a volume.
We leave open the question of whether or not, using our approach,
 a carefully designed field $\xi$ (perhaps constructed dynamically along side the solution) can be used to ensure quasisymmetry
on a flux surface or a volume.}

\section{Proof of Theorem \ref{noqsthmvague}}
\label{mainthmpf}

We start by giving an outline of the arguments used to establish
the main theorem. All details can be found in \cite{CDG20}.

Let $D_0, D, \psi_0, \xi, g$ be as in the statement of Theorem
 \ref{noqsthmvague}.
We will start by constructing a solution $\psi$ to the generalized
Grad-Shafranov equation \eqref{gGS} of the form $\psi = \psi_0\circ
\gamma^{-1}$ for a diffeormorphism $\gamma : D_0 \to D$ which is
to be determined. With the toroidal coordinates $(r, \theta, \varphi)$
defined as in \eqref{toroidalcoords}, for functions
$\eta, \phi$ independent of $\varphi$, write
$\nabla \eta = \pa_r \eta e_r + \frac{1}{r} \pa_\theta \eta e_\theta$
and $\nabla^\perp \phi = -\pa_r \phi e_\theta + \frac{1}{r} \pa_\theta
f e_r$. We will look for $\gamma(r, \theta) = (r, \theta)
+\nabla \eta(r, \theta) + \nabla^\perp \phi(r, \theta)$ and the
functions $\eta, \phi$ are the unknowns. For simplicity, using
the assumption that Vol $D$ = Vol $D_0$, we will require that
$\det \nabla \gamma^{-1} = 1$. After a short calculation, this condition
reads
\begin{equation}
 \Delta \eta = \mathcal{N}_{\eta}[\phi,\eta],
 \label{}
\end{equation}
where $\mathcal{N}_{\eta}[\phi, \eta] =
\mathcal{N}_\eta(\pa \phi, \pa \eta, \pa^2\phi,\pa^2 \eta)$ is a quadratic nonlinearity.
We will pose boundary conditions momentarily. We think
of this equation as determining $\eta$ at the linear level from
$\phi$ and it remains to determine $\phi$ in a such a way that
$\psi = \psi_0\circ \gamma^{-1}$ is a solution to
\eqref{gGS}. We now describe how this is done.

The Grad-Shafranov equation \eqref{axigs} is of the form
\begin{equation}
 L_0 \psi_0 = F_0(\psi_0) + G_0(r,\theta, \psi_0), \qquad \text{ in } D_0
 \label{L0model}
\end{equation}
with nonlinearities $F_0 = P'_0,$ $ G_0 = \frac{1}{R^2} C_0C_0'(\psi_0)$ and where the operator $L_0$ is elliptic.
Similarly, we write the generalized Grad-Shafranov \eqref{gGS}
in the form
\begin{equation}
 L \psi = F(\psi) + G(r,\theta, \psi), \qquad \text{ in } D.
 \label{Lmodel}
\end{equation}
At this stage, the function $F$ (which is related to the pressure of the solution in our application) is actually undetermined and will
be chosen momentarily, while $G$ can be chosen to be any function
sufficiently close to $G_0$.

A calculation (see Appendix B of \cite{CDG20})
shows that provided $\det \nabla \gamma^{-1} = 1$, we have
\begin{equation}
 (\nabla \psi)\circ \gamma^{-1}
 = \nabla \psi_0 +
 \nabla \pa_s\phi - \nabla^\perp \pa_s\eta
 + \nabla^\perp\phi\cdot \nabla^2\psi_0
 + \nabla \eta \cdot \nabla^2 \psi_0,
 \label{changeofcoords}
\end{equation}
where we have introduced the notation $\pa_s = \nabla \psi_0\cdot \nabla^\perp$ for the ``streamline derivative''.
 Then $\pa_s$ is tangent to level sets of $\psi_0$.
After a computation, composing both sides
of \eqref{Lmodel} with $\gamma^{-1}$ and using \eqref{changeofcoords},
\eqref{Lmodel} takes the form
\begin{equation}
 \mathscr{L}_{\psi_0} \pa_s\phi = \mathcal{N}_\phi[\phi, \eta] + F(\psi_0) - F_0(\psi_0)
 + G(r, \theta, \psi_0) - G_0(r, \theta, \psi_0),
 \label{Lpsi0model}
\end{equation}
where $\mathcal{L}_{\psi_0}$ is the linearization of
$L_0$ around $\psi_0$,
for a function $\mathcal{N}_\phi[\phi, \eta] = \mathcal{N}_\phi
(\pa \phi, \pa \eta,
\pa\pa_s\phi, \pa \pa_s\eta)$ (whose
explicit form can be found in Appendix
B of \cite{CDG20}), which consists of terms which
are linear in $\eta$ and its derivatives, and either
nonlinear or weakly linear in derivatives of $\phi$,
meaning it involves terms which can bounded by $\epsilon |\pa\phi|$, for example.
This latter point is a consequence of the assumption that
$\xi - \xi_0$ is sufficiently small.
Notice
that at the linear level this is an equation for $\pa_s\phi$ and not $\phi$ itself.
In order for this equation to be solvable for $\phi$ at the linear level
(given appropriate boundary conditions), there
are two requirements. The first is that
$\mathscr{L}_{\psi_0}$ should be invertible. The second is a somewhat
subtle condition which is easiest to understand in the simple model case.
In order to solve the problem
$
 \Delta \pa_\theta u = f,$
in the unit disk, say (with arbitrary boundary conditions), it is clearly
necessary that $\int_{0}^{2\pi} f \, d\theta = 0$. We will now impose
a condition on \eqref{Lpsi0model} which is analogous to this one and which
will determine the function $F$ at the linear level.
Assume that the Dirichlet problem for $\mathscr{L}_{\psi_0}$,
\begin{alignat}{2}
 \mathscr{L}_{\psi_0} u &= f, &&\qquad \text{ in } D_0,\\
 u &= 0, &&\qquad \text{ on } \pa D_0,
 \label{}
\end{alignat}
has a unique solution for $f \in L^2$, say. Writing
$\mathscr{L}^{-1}_{\psi_0} f = u$, if we apply
$\mathscr{L}^{-1}_{\psi_0}$
to both sides of \eqref{Lpsi0model} and integrating with respect to $ds$ over
the streamline $\{\psi_0 = c\}$ (considered as a subset of the two-dimensional
set $D_0$) we find
\begin{equation}
 0 = \oint_{\{\psi_0 = c\}} \mathscr{L}^{-1}_{\psi_0}\big(F(\psi_0) - F_0(\psi_0)\big)
 \, \rmd s
 + \oint_{\{\psi_0 = c\}} \mathscr{L}^{-1}_{\psi_0}\big(G - G_0 + \mathcal{N}\big) \, \rmd s
 \label{}
\end{equation}
This is an equation which must be solved for $F$.
Writing $T(c)q = \oint_{\{\psi_0 = c\}}
\mathscr{L}_{\psi_0}^{-1} q \, \rmd s$,
given $R$ depending only on the streamline
$R = R(c)$, we would like to be able to find
$q = q(c)$ with $T(c)q =R$. This is a complicated problem
which would be hard to address directly,
however in \cite{CDG20}, we show that such $q$ can be found,
assuming that the following hypotheses hold:
\\

\noindent  {\rm Hypothesis 1 ({\rm \textbf{H1}}): }
The operator $\mathscr{L}_{\psi_0}$ is positive definite.\\

 \noindent  {\rm Hypothesis 2 ({\rm \textbf{H2}}): } There exists a constant $C>0$ such that we have
 \be
  \mu(c) = \oint_{\{\psi_0=c\}} \frac{\rmd \ell}{|\nabla \psi_0|}
 \leq C  \qquad \text{for}\ \   c\in {\rm rang} (\psi_0),
 \ee
 where $\ell$ is the arc-length parameter.\\

Notice that the hypothesis ({\rm \textbf{H1}}) in
particular ensures that the operator $\mathscr{L}_{\psi_0}^{-1}$ is well-defined.
This is a condition on
$P_0, C_0$. Hypothesis ({\rm \textbf{H2}}) concerns the travel time $\mu(c)$ for a particle governed by the Hamiltonian system $\dot{x}= \nabla^\perp \psi_0(x)$ and moving along the streamline of $\{\psi_0=c\}$.  It is easy to see that it holds provided $\psi_0$ has at most one critical
point in $D_0 = T_0 \cap \{\varphi = 0\}$ and that it vanishes no faster than
to first order there. We remark that ({\rm \textbf{H2}}) is trivially
satisfied if $|\nabla \psi_0|$ is bounded below
in the domain $D_0$. This could be accomplished if, for example, one worked on a ``hollowed out" toroidal domain.

We now discuss the boundary conditions.
Assume that $D$ is the interior of a Jordan curve $B$,
\begin{equation}
 \pa D = \{p\in \R^2\ | \ \mathfrak{b}(p) = 0\}.
 \label{}
\end{equation}
We also write $\pa D_0 =\{p \in \R^2 \ |\  \mathfrak{b}_0(r,\theta) = 0\}$
where $\mathfrak{b}_0$ is chosen with $|\nabla \mathfrak{b}_0| = 1, \nabla
\psi_0\cdot \nabla \mathfrak{b}_0 > 0$.
We write $\gamma - {\rm id} = \nabla \eta + \nabla^\perp \phi = \alpha e_x + \beta e_y$ where
$(x,y)$ are rectangular coordinates.
Using that $\mathfrak{B}_0|_{\pa D_0} = 0$, the requirement that $\gamma:\pa D_0 \to\pa D$
can be written as
\begin{align}
 0 = \mathfrak{b}\circ \gamma|_{\partial D_0}  &= \mathfrak{b}_0\circ \gamma|_{\partial D_0} +
  ( \delta \mathfrak{b})\circ \gamma|_{\partial D_0} \\
 &=
\alpha \pa_x \mathfrak{b}_0|_{\partial D_0} + \beta \pa_y \mathfrak{b}_0|_{\partial D_0}
 + \mathfrak{b}_1(\alpha, \beta )|_{\partial D_0}
 \label{bc}
\end{align}
where $\delta \mathfrak{b} = \mathfrak{b} - \mathfrak{b}_0$, and
where the remainder $\mathfrak{b}_1$ is
\begin{equation}
 \mathfrak{b}_1(\alpha, \beta, x,y)
 =  \mathfrak{b}_0\circ \gamma - \mathfrak{b}_0   -\alpha \pa_1
 \mathfrak{b}_0 -\beta \pa_2 \mathfrak{b}_0 + ( \delta \mathfrak{b})\circ \gamma.
 \label{B1def}
\end{equation}
 Returning to $\phi, \eta$, we have
\begin{equation}
\alpha \pa_1 \mathfrak{b}_0 + \beta \pa_2 \mathfrak{b}_0
 =  \nabla \phi\cdot \nabla^\perp \mathfrak{b}_0 +  \nabla \eta \cdot\nabla \mathfrak{b}_0
 \label{alphaphi}
\end{equation}
By the choice of $\mathfrak{b}_0$, we have $\nabla \eta \cdot \nabla
\mathfrak{b}_0 =\partial_n \eta$ where $n$ is the outward-facing normal
to $\pa D_0$. Additionally using that $\psi_0$ is constant on the boundary we have
$\nabla^\perp \mathfrak{b}_0 = \frac{\nabla^\perp \psi_0}{|\nabla\psi_0|}$, and
using \eqref{alphaphi} and these observations, the formula
\eqref{bc} becomes
\begin{equation}
 \frac{1}{|\nabla \psi_0|} \pa_s\phi +\partial_n \eta  = -\mathfrak{b}_1(\phi, \eta),
 \qquad \text{ on } \pa D_0.
 \label{}
\end{equation}

This is one boundary condition for the two functions $\phi, \eta$.
Again we need to ensure that this equation is compatible with the
requirement $\oint_{\{\psi_0 = \psi_0|_{\pa D_0}\}} \pa_s\phi \, \rmd s = 0$.
We therefore take $\partial_n\eta$ constant on the boundary
and impose the following nonlinear boundary conditions.
\begin{alignat}{2}
\partial_{{n}}  \eta &= -\frac{ \oint_{\partial D_0}
   \mathfrak{b}_1(\phi, \eta)\, {\rmd\ell}}{{\rm length}(\partial D_0)}
 &&\quad \text{ on } \pa D_0,
 \label{etabc}\\
 \pa_s \phi &= |\nabla \psi_0| \left(-\mathfrak{b}_1(\phi, \eta) +\frac{ \oint_{\partial D_0}
   \mathfrak{b}_1(\phi, \eta)\, {\rmd\ell}}{{\rm length}(\partial D_0)} \right)
 &&\quad \text{ on } \pa D_0.
 \label{phibc}
\end{alignat}

We now summarize the result of the above calculation.
The function $\psi = \psi_0\circ \gamma^{-1}$
is a solution of the equation
\eqref{Lmodel} in $D$ with constant boundary value
provided the diffeomorphism $\gamma$ is of the
form $\gamma = {\rm id} + \nabla \eta +
\nabla^\perp \phi$ and the functions $\eta, \phi
:D_0 \to \R$
satisfy the elliptic equations
\begin{alignat}{2}
 \Delta \eta &= \mathcal{N}_\eta[\phi, \eta]
 &&\qquad \text{ in } D_0, \label{etaeqn}\\
 \mathscr{L}_{\psi_0} \pa_s\phi&=
 \mathcal{N}_\phi[\phi, \eta]+ F(\psi_0) - F_0(\psi_0)
 + G(\psi_0, r, \theta) - G_0(\psi_0, r, \theta),&&\qquad \text{ in } D_0,
 \label{phieqn}
\end{alignat}
where $F$ is determined by solving
\begin{equation}
  \oint_{\{\psi_0 = c\}}
  \mathscr{L}_{\psi_0}^{-1} F \, \rmd s
  =
 \oint_{\{\psi_0 = c\}} \mathscr{L}_{\psi_0}^{-1}
 (G_0 - G- \mathcal{N}_\phi + F_0) \, \rmd s,
 \label{Feqn}
\end{equation}
and where $\eta, \phi$ satisfy the boundary conditions
\eqref{etabc}-\eqref{phibc}.

This nonlinear system can be solved by the following iteration scheme.
Given $\eta^{N-1}, \phi^{N-1}$, define $F^N = F^N(c)$ by solving
\begin{equation}
 \oint_{\{\psi_0 = c\}} \mathscr{L}_{\psi_0}^{-1} F^N \, \rmd s
 = \oint_{\{\psi_0 = c\}} \mathscr{L}_{\psi_0}^{-1}
 (G_0 - G- \mathcal{N}_\phi[\phi^{N-1}, \eta^{N-1}] + F_0) \, \rmd s.
 \label{}
\end{equation}
Then solve for $\eta^N, \Phi^N$ satisfying
\begin{alignat}{2}
 \Delta \eta^{N} &= \mathcal{N}_\eta[\phi^{N-1}, \eta^{N-1}]
&&\qquad \text{ in } D_0,
  \label{etaeqnN}\\
 \mathscr{L}_{\psi_0} \Phi^{N} &= \mathcal{N}_\phi[\phi^{N-1},
 \eta^{N-1}]
 + F^N - F_0 + G - G_0
 &&\qquad \text{ in } D_0,
\end{alignat}
with boundary conditions
\begin{alignat}{2}
\partial_n \eta^N &= \int_{D_0} \mathcal{N}_{\eta}[\phi^{N-1}, \eta^{N-1}] \rmd x,
 \label{etaNbc}\\
 \Phi^N &= |\nabla \psi_0| \left(-\mathfrak{b}_1(\phi^{N-1}, \eta^{N-1}) +\frac{ \oint_{\partial D_0}
   \mathfrak{b}_1(\phi^{N-1}, \eta^{N-1})\, {\rmd\ell}}{{\rm length}(\partial D_0)} \right)
 &&\quad \text{ on } \pa D_0.
 \label{}
\end{alignat}

The boundary condition for $\eta^N$ has been chosen so that the Neumann
problem \eqref{etaeqnN}-\eqref{etaNbc} is solvable. Once $\Phi^N$ has been
found, as a consequence of the choice of $F^N$ it can be shown that
$\oint_{\{\psi_0 = c\}} \Phi^N \, \rmd s = 0$ for all $c$ and so $\Phi^N = \pa_s\phi^N$
for a function $\phi^N$ which is determined up to a constant, which can
be fixed throughout the iteration by requiring that $\int_{D_0} \phi^N = 0$.
In \cite{CDG20} we prove that this iteration converges in a suitable topology.
We remark that the boundary condition \eqref{etaNbc} is not the same as
the boundary condition in \eqref{etabc} but as a consequence of Vol $D_0 = $
Vol $D$, they agree after taking $N\to \infty$.

\begin{proof}[Proof of Theorems \ref{noqsthmGS}-\ref{noqsthmvague}]

We first reduce the problem to solving a certain elliptic problem
on the domain $D$. Given a (local) coordinate system
$(x_1, x_2)$ defined on a neighborhood of $D$, we can extend it to a (local) coordinate system
on a neighborhood of the torus $T$ by pulling back along the flow of $\xi$.
Explicitly, given $p \in T$, there is a unique $p_0
\in D$ and a unique smallest $x_3 > 0$ so that $\Phi_{x_3}(p_0) = p$
where $\Phi_s(p_0)$ denotes the time-$s$ flow of $\xi$ starting from
a point $p_0 \in D$, because the integral curves of $\xi$ are
closed. Then the map
$\Psi: T \to D \times \R$ defined by $\Psi(p) = (p_0, x_3)$ is a
(local)
diffeomorphism onto its image. In these coordinates,
$\xi \cdot \nabla f = \frac{\pa}{\pa x^3} f$ for any function $f$.
We now express the given metric $g$ in this coordinates,
$g = \sum_{i,j = 1}^3 g_{ij}(x_1,x_2,x_3) dx^idx^j$, where
$g_{ij} = g(\pa_{x^i}, \pa_{x^i})$. By the definition of the
Lie derivative of the metric we have
\begin{equation}
 (\L_\xi g)_{ij}
 = (\xi\cdot \nabla) g_{ij}
 - g([\xi, \pa_{x^i}], \pa_{x^j})
 - g([\xi, \pa_{x^j}], \pa_{x^i})
 = \frac{\pa}{\pa x^3} g_{ij} = 0,
 \label{}
\end{equation}
since by assumption $\L_\xi g = 0$ and since $[\xi, \pa_{x^\ell}] =
[\pa_{x^3}, \pa_{x^\ell}]= 0$ by
construction.
In this coordinate system, the Grad-Shafranov equation
is the following three-dimensional elliptic equation
\begin{alignat}{2}
 L \psi := \sum_{i,j = 1}^3
 a_{\xi,g}^{ij} \pa_{x^i}\pa_{x^j}\psi
 + \sum_{i = 1}^3 b_{\xi,g}^i \pa_{x^i} \psi
 + G_{\xi,g}(x_1, x_2, x_3, C, \psi)+ \frac{1}{\sqrt{|g|}}P'(\psi)
 &=0,
 &&\quad \text{ in } T,
 \label{3dgs}
\end{alignat}
for coefficients $a_{\xi,g}^{ij}, b_{\xi,g}^j$ and a function $G_{\xi,g}$, depending
on $x_1, x_2, x_3$ which are all computed explicitly in Appendix
\ref{explicitGS}.
The crucial point is that all of these quantities are independent of
$x_3$, because they involve algebraic functions of components of the metric.

We can therefore look for a two-dimensional solution,
$\overline{\psi} = \overline{\psi}(x_1, x_2)$, of the equation
\begin{equation}
 \sum_{i,j = 1}^2
 a_{\xi,g}^{ij}(x_1, x_2) \pa_{x^i} \pa_{x^j} \overline{\psi}
 + \sum_{i = 1}^2 b_{\xi, g}^i(x_1, x_2) \pa_{x^j} \overline{\psi}
  +G_{\xi,g}(x_1, x_2, C, \overline{\psi}) + \frac{1}{\sqrt{|g|}}P'(\overline{\psi}) = 0,
 \label{2dgspbm}
\end{equation}
in $D$
with $\overline{\psi}$ constant on $\pa D$.
Given such $\overline{\psi}$, we can recover $\psi$ satisfying \eqref{3dgs}
by setting $\psi(x_1, x_2, x_3) = \overline{\psi}(x_1, x_2)$, i.e. by extending
$\overline{\psi}$ to be constant along integral curves of $\xi$.
Since the
integral curves of $\xi$ are closed it follows that $\psi$ is as smooth as
$\overline{\psi}$, and by construction we have $\L_\xi \psi = 0$. Also since $\xi$ is tangent to $\pa T$
it follows that the resulting $\psi$ is constant there.

Supposing that we have a solution $\overline{\psi}$ as above,
by Lemma \ref{propBlem1},
defining $B$ as in \eqref{gBqs}
provides a magnetic field satisfying $\div B = 0$ and which satisfies
the MHS equations with respect to $g$, $\curl_g B \times_g B_g = \nabla_g P$
exactly, by
Lemma \ref{curlBBlem}. As a consequence,
$B$ satisfies the usual MHS equations with forcing
\begin{equation}
 f = (\curl_g - \curl) B \times B + \curl_g B(\times_g - \times) B
 + (\nabla_g - \nabla)P
 \label{}
\end{equation}
From the formulae for $\curl_g, \times_g$
in Appendix \ref{notation}, it is clear this
satisfies a bound of the form \eqref{mhdwforcebd},
completing the proof of Theorem \ref{noqsthmvague}.

We have therefore reduced the problem to solving the generalized
Grad-Shafranov equation for a function $\overline{\psi}: D \to \R$.
In what follows we will abuse notation and just write $\psi = \overline{\psi}$.
Using e.g. variational methods (Proposition
11.4 of  \cite{T10}), in principle one can find a weak
 solution to this equation in $H^1_0$.  Unfortunately, these solutions need not be smooth and, more importantly, the structure of the level sets cannot be specified. In particular, the flux function may possess ``magnetic islands".  We
provide here am explicit construction of classical solutions which allows for control of flux surfaces, based on the approach
of \cite{CDG20}. As explained above, the method there is to deform
a solution $\psi_0$ to the axisymmetric Grad-Shafranov equations
into a solution to \eqref{2dgspbm} and this has the benefit of ensuring
that the level sets of the resulting $\psi$ are tori, as well as providing
a simple algorithm to compute the solution.

Let $\psi_0:D_0 \to \R$ be a solution
to the axisymmetric Grad-Shafranov equation
\eqref{axigs} with $\psi_0|_{\pa D_0}$ constant,
satisfying the mild hypotheses (\textbf{H1}) and (\textbf{H2})
(see Appendix \ref{examplesec} for an example of such a flux function).
We begin by writing the axisymmetric Grad-Shafranov equation on the
 unit disk $D_0$ in the same coordinate system $(x_1, x_2)$
  as above. Letting $h_{ij}$ denote the components of the Euclidean
  metric restricted to $D_0$ in this coordinate system, on $D_0$ we have
 \begin{equation}
 L_0\psi_0
 := \sum_{i,j = 1}^2 \tilde{a}^{ij}_{\xi_0,\delta}(x_1, x_2) \pa_{x^i} \pa_{x^j}\psi_0
 + \sum_{i = 1}^2 \tilde{b}^i_{\xi_0,\delta}(x_1, x_2) \pa_{x^i} \psi_0 +
 \tilde{G}_{\xi_0,\delta}(x_1, x_2, C_0, \psi_0) + \frac{1}{\sqrt{|h|}} P_0'(\psi_0)
 =0,
 \label{}
\end{equation}
where $|h|$ denotes the determinant of the matrix $h_{ij}$,
$\xi_0 = R e_\Phi$ is the generator of rotations in the
$Z = 0$ plane, and $\tilde{a}^{ij}_{\xi_0,\delta}, \tilde{b}^i_{\xi_0,\delta},
\tilde{G}_{\xi_0,\delta}$ can be computed
explicitly by changing variables in
\eqref{axigs},
\begin{align}
    \tilde{a}^{ij}_{\xi_0,\delta} &= \frac{\sqrt{|h|}}{|\xi_0|^2}
    h^{ij},\quad
    \tilde{b}^i_{\xi_0,\delta} =
     \sum_{j =1,2} \frac{1}{\sqrt{|h|}} \pa_{x^i} \bigg( \frac{\sqrt{|h|}}{|\xi_0|^2}
     h^{ij}\bigg),\quad
     \tilde{G}_{\xi_0,\delta} &=\frac{C(\psi)}{|\xi_0|^2}\left(\frac{C'(\psi)}{ \sqrt{|h|}}
    -
    \xi_0\cdot \curl \xi_0\right),
    \label{}
   \end{align}
In order to appeal to the results of \cite{CDG20} we need that
the coefficients of $L$ are close to those of $L_0$, that $G_{\xi_0,\delta}$ is close to $G_{\xi, g}$ and
that the domain $D$ is close to $D_0$. For simplicity,
we take the function $C$ in \eqref{2dgspbm} to just be $C_0$ though
this is not essential.
From the
 formulas in Appendix \ref{explicitGS} we have
\begin{equation}
  \sum_{i,j =1}^2\|a_{\xi,g}^{ij} - \tilde{a}_{\xi_0,\delta}^{ij}\|_{C^{k,\alpha}}
  +\sum_{i = 1}^2\|b_{\xi,g}^i - \tilde{b}_{\xi_0,\delta}^i\|_{C^{k,\alpha}}
  + \|G_{\xi,g} - \tilde{G}_{\xi_0,\delta}\|_{C^{k,\alpha}}
  \leq c\|g - \delta\|_{C^{k+1,\alpha}}
  +c\|\xi - \xi_0\|_{C^{k+1,\alpha}}
 \label{controlcoeffs}
\end{equation}
where $c$ is a constant depending on $k, \alpha,
\sum_{i,j = 1}^3\|g\|_{C^{k+2,\alpha}}$,
 and $\|C_0\|_{C^{k+3,\alpha}}$. Here, and in what follows,
we are writing $C^{k+2,\alpha} = C^{k+2,\alpha}(U)$ where
$U$ is a domain containing both $D$ and $D_0$.
By Theorem 3.1 from \cite{CDG20}, there is $\epsilon > 0$
depending on $k, \alpha, \psi_0, D_0$ so that if the
following holds,
\begin{equation}
  \sum_{i,j =1}^2\|a_{\xi,g}^{ij} - \tilde{a}_{\xi_0,\delta}^{ij}\|_{C^{k,\alpha}}
  +\sum_{i = 1}^2\|b_{\xi,g}^i - \tilde{b}_{\xi_0,\delta}^i\|_{C^{k,\alpha}}
  + \|G_{\xi,g} - \tilde{G}_{\xi_0,\delta}\|_{C^{k,\alpha}}
  + \|\mathfrak{b} - \mathfrak{b_0}\|_{C^{k+2,\alpha}}
  \leq \epsilon,
 \label{needforflex}
\end{equation}
and the hypotheses ({\rm \textbf{H1}}) and ({\rm \textbf{H2}})
hold, there is a function
$\psi \in C^{k,\alpha}$ of the form $\psi = \psi_0\circ \gamma^{-1}$ where
 $\gamma: D_0 \to D$ is a diffeomorphism and $\psi$ satisfies the
 generalized Grad-Shafranov equation \eqref{gGS} for some pressure profile
 $P$ which is close to $P_0$.
 We now take $\|g - \delta\|_{C^{k+1,\alpha}},\|\xi - \xi_0\|_{C^{k+2,\alpha}}$ and
 $\|\mathfrak{b} - \mathfrak{b}_0\|_{C^{k+2,\alpha}}$
 small enough that \eqref{needforflex} holds and let $\psi$ be the flux
 function guaranteed by Theorem 3.1 from \cite{CDG20}. This completes the proof of Theorem \ref{noqsthmGS}.
\end{proof}

\appendix

\section{Flux function satisfying our hypotheses}
\label{examplesec}

 The purpose of this section is to give a simple example of
 a flux function satisfying the hypotheses ({\rm \textbf{H1}})-({\rm \textbf{H2}}).
 We will work in toroidal coordinates $(r, \theta, \varphi)$ defined by
 \begin{equation}
  R = R_0 + r\cos \theta, \qquad
  Z = r\sin \theta, \qquad \Phi = \varphi,
  \label{toroidalcoords}
 \end{equation}
 where $(R, Z, \Phi)$ are the usual cylindrical coordinates on $\R^3$. In these
 coordinates, \eqref{axigs} becomes
 \begin{equation}
   \pa_r^2 \psi_0 + \frac{1}{r} \pa_r\psi_0 + \frac{1}{r^2}
  \pa_\theta^2 \psi_0 - \frac{1}{R} \bigg(\cos \theta \pa_r \psi_0
  - \frac{\sin \theta}{r} \pa_\theta \psi_0\bigg)
  + R^2p_0'(\psi_0) + C_0C_0'(\psi_0) = 0.
  \label{toroidalgs}
 \end{equation}
 The flux function we exhibit is not an exact solution of \eqref{toroidalgs} but satisfies
 it when the aspect ratio of the torus is taken to infinity.  Using Theorem 3.1 from \cite{CDG20}, one can show that there exist solutions on the true axisymmetric torus with large aspect ratio nearby this example.  Although they do not have a simple analytical form, they will continue to satisfy ({\rm \textbf{H1}})-({\rm \textbf{H2}}) as these are open conditions.

We consider the torus where $r$  ranges in $ [0, r_0]$ with $0< r_0<R_0$
for $R_0 > 1$ and solve the equation \eqref{toroidalgs} with the choices
\be
{\psi}_0(r)= \bar{\psi}\Big(1-(r/r_0)^2\Big), \qquad C_0(\psi) = \bar{c}\sqrt{\bar{\psi}- \psi + \epsilon },  \qquad P_0(\psi)=\bar{p} (r_0 R_0)^{-2} \psi,
\ee
for $\epsilon \ll 1$ (this is to regularize the square-root) and for some constants $\bar{\psi}, \bar{c}$ and $\bar{p}$.   The functions $C_0$ and $P_0$ are both infinitely differentiable functions of $\psi$. Note that the pressure vanishes at the outer boundary where $\psi_0$ is zero, and so this boundary may be interpreted as vacuum.
For special choices of constants, ${\psi}_0$ solves the ``infinite aspect ratio'' Grad-Shafranov equation (\eqref{toroidalgs} as $R_0\gg 1$)
\be
  \pa_r^2\psi_0 + \frac{1}{r^2} \pa_\theta^2\psi_0 = -R_0^2 P_0'(\psi_0) - C_0C_0'(\psi_0),
\ee
since $  \pa_r^2\psi_0 = -2\bar{\psi}/r_0^2$, $R_0^2 P_0'(\psi_0)=  \bar{p}/r_0^2$ and  $C_0C_0'(\psi_0)=  \bar{c}^2$. Thus $\psi_0$ is a solution if
$
  \bar{p} = 2\bar{\psi}- (\bar{c}r_0)^2.
$

\section{Geometric identities}
\label{notation}

In this section we recall some basic definitions and facts from Riemannian
geometry which will be used in the upcoming sections.
These are standard and we include the details for the convenience
of the reader.
Throughout we fix a Riemannian metric $g$. In our applications, we will take
either $g = \delta$, the Euclidean metric, or $g$ will be a metric
with $\L_\xi g = 0$ for a given vector field $\xi$.
We let $\flat, \sharp$ denote the usual operations of lowering and raising
indices with respect to $g$. If $X = X^i \pa_{x^i}$ is a vector field and
$\beta = \beta_i dx^i$ is a one-form, where $\{x^i\}_{i = 1}^3$ are arbitrary
local coordinates, then
\begin{equation}
 X^\flat = g_{ij}X^j dx^i, \qquad
 \beta^\sharp = g^{ij}\beta_j \pa_i.
 \label{flat}
\end{equation}

We write $\nabla_g f$ for the gradient of $f$ with respect to the metric
$g$,
\begin{equation}
  \nabla_g f = (df)^\sharp, \qquad
  (\nabla_g f)^i = g^{ij} \pa_j f.
 \label{nablagdef}
\end{equation}

In an arbitrary coordinate system $\{x^i\}_{i = 1}^3$,
if $X = X^i \pa_i$ is a vector field and $\beta = \beta_i dx^i$
is a one-form then $\nabla X, \nabla \beta$ have components
\begin{equation}
 \nabla_i X^j = \frac{\pa}{\pa x^i} X^j + \Gamma_{ik}^j X^k,
 \qquad
 \nabla_i \beta_j = \frac{\pa}{\pa x^i} \beta_j - \Gamma_{ij}^k \beta_k,
 \label{covariant}
\end{equation}
where $\Gamma^i_{jk}$ are the Christoffel symbols in this
coordinate system, defined by
\begin{equation}
 \Gamma_{ij}^k = \frac{1}{2} g^{k\ell} \big( \pa_i g_{j\ell} + \pa_j g_{i\ell}
 - \pa_\ell g_{ij}\big).
 \label{christ}
\end{equation}
Here we are writing $g_{ij}$ for the components of the metric in this
corodinate system and $g^{ij}$ for the components of the inverse metric.
The $\Gamma$ are symmetric in the lower
indices,
\begin{equation}
 \Gamma_{ij}^k = \Gamma_{ji}^k.
 \label{symmetry}
\end{equation}
We also note that covariant differentiation commutes with lowering and raising indices since
\begin{equation}
 \nabla_i g_{jk} = \nabla_i g^{jk} = 0.
 \label{metric}
\end{equation}
Let us also recall that the divergence of a vector field can be written as
\begin{equation}
 \div_g X = \nabla_i X^i = \frac{1}{\sqrt{|g|}}\pa_i (\sqrt{|g|} X^i),
 \label{}
\end{equation}
where $|g|  =\det g$ denotes the determinant of the matrix
with components $g_{ij}$.

We let $\L_X$ denote the Lie derivative
in the direction $X$. If $f$ is a function then $\L_X f$ is defined by
\begin{equation}
 \L_X f = X^i\pa_i f = X f.
 \label{}
\end{equation}
For a vector field $Y$, $\L_X Y$ is the commutator $\L_XY = [X, Y]$.
In an arbitrary coordinate system, $\L_X Y = (\L_X Y)^i\pa_i$ with
\begin{equation}
(\L_X Y)^i = X^j \pa_j Y^i - Y^j\pa_j X^i.
 \label{}
\end{equation}
Many of our results will be stated in terms of the deformation tensor
of $X$, denoted $\L_X g$, which is the $(0,2)$ tensor defined by the formula
\begin{equation}
  X \big( g(Y, Z) \big) =
 (\L_X g)(Y, Z) + g(\L_X Y, Z) + g(Y, \L_X Z)
 \label{product1}
\end{equation}
In an arbitrary coordinate system, $\L_X g = \L_X g_{ij} dx^i dx^j$ and
a standard calculation shows that
\begin{equation}
 \L_X g_{ij} = \nabla_i X_j + \nabla_j X_i, \qquad X_k = g_{k\ell} X^\ell,
 \label{deformation}
\end{equation}
where $\nabla$ denotes covariant differentiation \eqref{covariant}.
We will often abuse notation and write $\L_X g(Y,\cdot)$ for the vector
field with components
\begin{equation}
 (\L_X g(Y, \cdot))^i = g^{ij} (\nabla_j X_k + \nabla_k X_j) Y^k.
 \label{abuse}
\end{equation}

Let $*_g$ denote the Hodge star with respect to the Riemannian volume
form $d\mu = \sqrt{|g|} dx^1 \wedge dx^2\wedge dx^3$. For the general definition
 see \cite{L13}. For our purposes we will only need to compute
 $*_g \omega$ when $\omega$ is a two-form. With $\epsilon_{ijk}$ denoting the Levi-Civita
symbol, so that $\epsilon_{ijk}$ denotes the sign of the permutation
taking $(1,2,3)$ to $(i,j,k)$, we have
\begin{equation}
 *_g (dx^i \wedge dx^j)
 = \sqrt{|g|} g^{ik}g^{j\ell} \epsilon_{k\ell m} dx^m.
 \label{}
\end{equation}
If $\beta = \beta_{ij} dx^i\wedge dx^j$ is a two-form then from the above
formula,
\begin{equation}
 *_g \beta = \sqrt{|g|} \beta^{k\ell} \epsilon_{k\ell m}\, dx^m,
 \qquad \beta^{k\ell} = g^{ik}g^{j\ell} \beta_{ij}.
 \label{}
\end{equation}
Let $d$ denote exterior differentiation. If $\beta$ is a one-form then
$d \beta$ is defined by
\begin{equation}
 d \beta = \pa_i \beta_j dx^i\wedge dx^j.
 \label{}
\end{equation}
We will use the following identity relating $*_g, d$ and covariant differentiation
$\nabla$. If $\omega =\omega_{ij} dx^i dx^j$ is a (0,2)-tensor then
\begin{equation}
  *_g d *_g \omega = \delta_g \omega,
\qquad  (\delta_g \omega)_i:=  g^{kj}\nabla_j \omega_{ik}. \label{deltag2form}
\end{equation}

  Given vector fields $X, Y$, let $X \times_g Y$ be the vector field
\begin{equation}
 X \times_g Y = (*_g X^\flat \wedge Y^\flat)^\sharp.
 \label{timesgdef}
\end{equation}
Explicitly, $X \times_g Y = (X \times_g Y)^\ell \pa_\ell$ with
$
 (X \times_g Y)^k = \sqrt{|g|} g^{k\ell} \epsilon_{ij\ell} X^i Y^j.
 $
The curl of a vector field, $\curl_g X$, is then defined by
\begin{equation}
 \curl_g X = ( *_g d X^\flat)^\sharp,
 \label{curlgdef}
\end{equation}
or, in components,
\begin{equation}
 (\curl_g X)^m
 = \sqrt{|g|} g^{mn} g^{ik} g^{j\ell}\epsilon_{k\ell n} \pa_i X_j
 = \sqrt{|g|} g^{mn} g^{ik} \epsilon_{k\ell n} \nabla_i X^\ell,
 \label{explicitcurl}
\end{equation}
where the second equality follows from a direct calculation involving
the formula for the Christoffel symbols \eqref{christ}.

We now collect some basic vector calculus identities.
\begin{lemma}\label{lemgid}
   Define $\times_g$ by \eqref{timesgdef}, $\curl_g$ by \eqref{curlgdef},
   and $\nabla_g$ by \eqref{nablagdef}. Suppose that $M$ is a subset of
   $R^3$.
 Let $\times, \cdot$ denote the usual cross and dot products in Euclidean space.
 Then we have
 \begin{align}
  (X \times_g Y)\cdot_g Z &= \sqrt{|g|} X\times Y \cdot Z,\label{determinant}\\
  \curl_g (f X) &= \nabla_g f \times_g X + f \curl_g X, \label{product}\\
  \curl_g \nabla_g f &= 0,
  \label{d2iszero}\\
  (X\times_g Y)\times_g Z &=
   (X\cdot_g Y) Z - (X\cdot_g Z )Y.
  \label{tripleproduct}
 \end{align}
\end{lemma}
\begin{proof}
The first three identities are immediate.
The last identity is proven by changing coordinates as in the proof of the upcoming
identity \eqref{timestimesglem} and we omit the proof.
\end{proof}

We will also need the following slightly more complicated identities
in the next section.
\begin{lemma}\label{lemmag} Let  $\nabla_g$ denote covariant differentiation.
For any vector fields $X, Y$,
 \begin{align}
  \curl_g (X\times_g Y) &= X \div_g Y - Y \div_g X + \L_Y X,
  \label{magic1}\\
   \nabla_g |X|_g^2 &= 2\nabla_X X - 2X \times_g \curl_g X,
  \label{magic2}\\
  \L_X (X^\flat) &= (\nabla_X X)^\flat + \frac{1}{2} \nabla_g |X|_g^2,
  \label{magic2point5}\\
  \div_g (X \times_g Y) &= Y\cdot_g \curl_g X - X \cdot_g \curl_g Y,
  \label{magic3}\\
  X \times_g \curl_g X &= \nabla_g |X|^2 + (\L_\xi g(X,\cdot))^\sharp,
  \label{lastid}
 \end{align}
  where $ \nabla_X := X^i \nabla_i$,  and where $\L_Xg$, defined in \eqref{deformation}
 denotes the deformation tensor of $X$, and we are using the notation \eqref{abuse}.
\end{lemma}

\begin{proof}
 We begin by writing
\begin{equation}
 (\curl_g (X \times_g Y))^\flat =
 *_g d *_g (X^\flat \wedge Y^\flat)
 =  \delta_g (X^\flat \wedge Y^\flat).
 \label{}
\end{equation}
Using \eqref{deltag2form} and writing $X_i = g_{ij} X^j, Y_k = g_{k\ell} Y^\ell$,
\begin{equation}
 \delta_g (X^\flat \wedge Y^\flat)_i =
 g^{kj}\nabla_j (X_i Y_k - X_k Y_i)
 = X_i g^{kj} \nabla_j Y_k - Y_i g^{kj} \nabla_j X_k
 + Y_k g^{kj} \nabla_j X_i - X_k g^{kj}\nabla_j Y_i.
 \label{}
\end{equation}
The first two terms are $X_i \div_g Y - Y_i \div_g X$. If we raise the
index on the last two terms (using \eqref{metric}) and use \eqref{symmetry} then we see
\begin{equation}
  Y_k g^{kj} \nabla_j X^i - X_k g^{kj}\nabla_j Y^i
  = Y^j \nabla_j X^i - X^j \nabla_j Y^i
  = Y^j \pa_j X^i - X^j \pa_j Y^i,
 \label{}
\end{equation}
which gives \eqref{magic1}.
To prove \eqref{magic2} we start by computing
$X \times_g \curl_g X$. Writing $\beta = X^\flat$,
a direct calculation using \eqref{explicitcurl} shows that
\begin{equation}
 \beta \wedge (*_g d \beta)
 = \sqrt{|g|} g^{ik} g^{j\ell} \epsilon_{k\ell m} \nabla_i X_j X_n \, dx^n\wedge
 dx^m
 \label{}
\end{equation}
and that
\begin{equation}
 *_g\big( \beta \wedge (*_g d \beta)\big)
 = |g| g^{ik} g^{j\ell} g^{nr} g^{mq}
 \epsilon_{k \ell m} \epsilon_{r q p} \nabla_i X_j X_n \, dx^p
 = |g| g^{mq} \epsilon_{k\ell m} \epsilon_{rqp} \nabla^k X^\ell X^r  dx^p.
 \label{}
\end{equation}
The identity \eqref{magic2} follows at any given point $P$ after changing
coordinates near $P$ so that expressed in these coordinates, the metric
is given by $\textrm{diag}(1,1,1)$.

To prove \eqref{magic2point5}, we write
\begin{equation}
 (\L_X X^\flat)_j = X^i \pa_i X_j + X_i \pa_j X^i = X^i \pa_i X_j
 + \frac{1}{2} \pa_j \big( g_{i\ell} X^\ell X^i\big)
 - \frac{1}{2} \big(\pa_j g_{i\ell}\big)X^\ell X^i,
 \label{}
\end{equation}
and so it follows from the definition of the covariant derivative that
\begin{equation}
 \L_X X_j - \nabla_X X_j - \frac{1}{2} \pa_j|X|_g^2
 = \Gamma_{ij}^k X^i X_k -\frac{1}{2} \big(\pa_j g_{i\ell}\big)X^\ell X^i
 \label{}
\end{equation}
and expanding the definition of the Christoffel symbols, the right-hand side is
\begin{align}
 \Gamma_{ij}^k X^i X_k -\frac{1}{2} \big(\pa_j g_{i\ell}\big)X^\ell X^i
 &=
 \frac{1}{2} X^i X^\ell \big(\pa_i g_{\ell j} + \pa_j g_{\ell i} - \pa_\ell g_{ij}\big)
 -\frac{1}{2} \big(\pa_j g_{i\ell}\big)X^\ell X^i
 = 0.
 \label{}
\end{align}
To prove \eqref{magic3}, we use \eqref{metric} and write
\begin{align}
 \div_g (X\times_g Y)
 = \nabla_i (\sqrt{|g|} g^{ij} \epsilon_{jk\ell} X^k Y^\ell)
 &= Y^\ell (\sqrt{|g|} g^{ij} \epsilon_{jk\ell} \nabla_i X^k)
 + X^k (\sqrt{|g|} g^{ij} \epsilon_{jk\ell} \nabla_i Y^\ell)\\
 &= Y^k (\curl_g X)_k- X^k (\curl_g Y)_k.
\end{align}
The final identity \eqref{lastid} follows from \eqref{magic2}.
\end{proof}

The following identity involving $\times$ and $\times_g$
is crucial for proving that the vector field
${B}_g$ defined in \eqref{gBqs} possesses a flux function. This result follows directly from standard vector calculus identities
when $g = \delta$.
\begin{lemma}
  \label{timestimesglem}
 If $\varphi$ is a function with $\L_X \varphi = 0$, then with $\nabla $ denoting
 the Euclidean gradient,
 \begin{equation}
  X \times (X \times_g \nabla_g \varphi) = -\frac{1}{\sqrt{|g|}}|X|_g^2  \nabla \varphi.
  \label{timestimesg}
 \end{equation}
\end{lemma}
\begin{proof}
  This follows from a straightforward but tedious argument; we include the details for the convenience of the reader.
  With $\omega$ the quantity on the left-hand side of \eqref{timestimesg},
  from the definitions we have
  \begin{equation}
   \omega^i = \sqrt{|g|} \delta^{ij}g^{\ell m} g^{pq} \epsilon_{jk\ell}\epsilon_{mnp}
   X^k X^n \pa_q \varphi.
   \label{omegaformula}
  \end{equation}

  Fix any $P \in M$ and choose coordinates
  $(Z^1, Z^2, Z^3)$ near $P$ so that at $P$, we have
  \begin{equation}
   g_{ij} \frac{\pa x^i}{\pa Z^\alpha} \frac{\pa x^j}{\pa Z^\beta} = \delta_{\alpha \beta}.
   \label{normalcoords}
  \end{equation}
   We note the following
  relation which will be useful in what follows: at $P$, we have
  \begin{equation}
   g^{\ell m} = \delta^{\alpha\beta}\frac{\pa x^\ell}{\pa Z^\alpha}\frac{\pa x^m}{\pa Z^\beta}.
   \label{inverse}
  \end{equation}

  Expressing $\omega$ in
  these coordinates, $\omega = \omega^a \pa_{Z^a}$ with $\omega^a = \frac{\pa Z^a}{\pa x^i}\omega^i$,
  evaluating at $P$ and using \eqref{inverse} to re-write $g^{\ell m}$ and
  $g^{ p q}$, from \eqref{omegaformula} we have
\begin{align}
 \omega^a
 &= \sqrt{|g|}\frac{\pa Z^a}{\pa x^i} \delta^{ij}g^{\ell m} g^{pq} \epsilon_{jk\ell}\epsilon_{mnp}
 X^k X^n \pa_q \varphi\\
 &= \sqrt{|g|} \frac{\pa Z^a}{\pa x^i}
 \frac{\pa x^\ell}{\pa Z^b} \frac{\pa x^m}{\pa Z^c} \frac{\pa x^p}{\pa Z^d}
 \frac{\pa x^k}{\pa Z^{b'}} \frac{\pa x^n}{\pa Z^{c'}} \delta^{bc} \delta^{dd'} \epsilon_{jk\ell}\epsilon_{mnp}
  X^{b'} X^{c'}
 \pa_{Z^{d'}}\varphi,
 \label{omegaa}
\end{align}
writing e.g. $X^{k} = \frac{\pa x^k}{\pa Z^{b'}} X^{b'}$.
Now we note that
\begin{align}
 \epsilon_{mnp} \frac{\pa x^m}{\pa Z^c} \frac{\pa x^n}{\pa Z^{c'}} \frac{\pa x^p}{\pa Z^d}
  &= \epsilon_{cc'd} \det (\pa_Z x),\qquad
  \epsilon_{jk\ell}  \frac{\pa x^j}{\pa Z^e} \frac{\pa x^k}{\pa Z^{b'}}
  \frac{\pa x^\ell}{\pa Z^b} = \epsilon_{eb'b}\det(\pa_Z x).
 \label{}
\end{align}
Indeed, the quantity on the left-hand side of e.g. the first
equality is antisymmetric in all three indices and so is a multiple of
$\epsilon_{cc'd}$ and evaluating at $c =1, c' = 2, d = 3$ gives the
result. Therefore \eqref{omegaa} reads
\begin{align}
 \omega^a &= \sqrt{|g|} \det(\pa_Zx)^2 \delta^{ij} \frac{\pa Z^a}{\pa x^i} \frac{\pa Z^e}{\pa x^j}
  \delta^{dd'}\delta^{bc}\epsilon_{cc'd}\epsilon_{beb'}
 X^{b'} X^{c'} \pa_{Z^{d'}} \varphi\\
 &=\sqrt{|g|} \det(\pa_Zx)^2 \delta^{ij} \frac{\pa Z^a}{\pa x^i} \frac{\pa Z^e}{\pa x^j}
  \delta^{dd'}
  \big( \delta_{c'e}\delta_{db'} - \delta_{c'b'}\delta_{de}\big)X^{b'} X^{c'} \pa_{Z^{d'}} \varphi,
 \label{}
\end{align}
using a well-known identity for the Levi-Civita symbol. Now we note that
\begin{equation}
 \delta^{dd'}\delta_{db'} \pa_{Z^{d'}}\varphi X^{b'} = X\cdot \nabla \varphi
  = 0,
 \label{}
\end{equation}
by assumption, and so
\begin{equation}
 \omega^a = -\sqrt{|g|} (\det \pa_Z x)^2 \delta^{ij} \frac{\pa Z^a}{\pa x^i}
 \frac{\pa Z^e}{\pa x^j} \pa_{Z^e} \varphi \delta_{c'b'} X^{c'} X^{b'}
 =-\sqrt{|g|} (\det \pa_Z x)^2 \delta^{ij} \frac{\pa Z^a}{\pa x^i}
 \pa_{x^j} \varphi |X|_g^2.
 \label{}
\end{equation}
From \eqref{normalcoords}, we have
$\sqrt{ |g|} (\det \pa_{Z}x)^2 = \frac{1}{\sqrt{|g|}}$ and so
at $P$ we find
\begin{equation}
 \omega^i = -\frac{1}{\sqrt{ |g| }} \delta^{ij} \pa_{x^j} \varphi |X|_g^2,
 \label{}
\end{equation}
and since $P$ was arbitrary we get the result.
\end{proof}

We finally record some useful formulae involving Lie derivatives along $g$ Killing fields.
We have

\begin{lemma}\label{crossLem}
Let $X,Y$ be vector fields and let $\xi$ be a Killing vector field for the metric $g$.  Then
\begin{align}\label{prodrule}
\L_\xi (X \times_g Y) &=   \L_{\xi} X \times_g Y + X \times_g \L_{\xi} Y,\\
 \label{curlcomm}
\L_{\xi} \curl_g X &= \curl_g  \L_{\xi}X .
\end{align}
\end{lemma}
\begin{proof}
To prove \eqref{prodrule}, we start  from the following fact,
which can be found on page 177 of \cite{F06}. If $\xi_0$ is a Killing field for a metric $g$,
$\L_{\xi}g = 0$, then
$
 \L_{\xi} *_g \alpha = *_g\L_{\xi} \alpha.
$
Similarly,
$
 \L_{\xi} X^\flat = (\L_{\xi} X)_{\flat},
 \label{}
$
where $\flat$ denotes lowering indices with $g$.
For any vector field $\xi$, if $\Phi_s$ denotes
its flow, we have the following identity
$
 \Phi_s^* *_g \alpha = *_{\Phi_s^*g} \Phi_s^* \alpha.
$
If $\xi$ is a Killing field for $g$ then this becomes
$
 \Phi_s^* *_g = *_g \Phi_s^* \alpha.
$
Differentiating this
at $s = 0$ and using the definition of the Lie derivative gives the result.
Now, to get the formula for $\L_\xi (X \times_g Y)$ we then recall that
$
 (X \times_g Y)_{\flat} =
 *_g(X^\flat \wedge Y^\flat).
$
Using  that $\L_{\xi}$
commutes with $\sharp$, $\flat$, and $*_g$,
\begin{align}
 \L_{\xi} (X \times_g Y)
 = \L_{\xi} *_g\big( X^\flat \wedge Y_{\flat})^\sharp
 = *_g\big(\L_{\xi} (X_{\flat} \wedge Y_{\flat} )\big)^\sharp
 &= *_g\big( (\L_{\xi}X)^\flat \wedge Y_{\flat} + X_{\flat}
 \wedge (\L_{\xi} Y)_{\flat}\big)^\sharp\\
 &=
 \L_{\xi} X \times_g Y + X \times_g \L_{\xi} Y.
 \label{}
\end{align}
To prove \eqref{curlcomm}, recall that
$
 \curl_g X = (*_g d X^\flat)^\sharp,
$
where $\flat$ and $\sharp$ denote
lowering and raising the index with $g$. Recall that if $\L_{\xi} g= 0$
then
\begin{equation}
 \L_{\xi} *_g = *_g\L_{\xi}, \qquad
 \L_{\xi} X^\flat = (\L_{\xi} X^\flat),
 \qquad
 \L_{\xi} \alpha^\sharp = (\L_{\xi}\alpha)^\sharp.
 \label{commute}
\end{equation}
After lowering the index on $\curl_g X$ and using the fact that Lie derivatives
commute with exterior differentiation, $\L_{\xi} d = d \L_{\xi}$, we obtain
$
 \L_{\xi} *_g (d X_{\flat})
 = *_g d (\L_{\xi} X)_{\flat}.
$
Raising the index with $g$ and using \eqref{commute} again we get the result.
\end{proof}

\section{Generalized quasisymmetric Grad-Shafranov equation}\label{appendproof}
\label{structural}

In this section we summarize the relationship between quasisymmetry and the MHS equation \eqref{steadymhd}.
Recall that the deformation tensor $\L_\xi \delta$ is defined by
\be  \label{Ldel}
(\L_\xi \delta  )(X,Y) =X\cdot ( \nabla \xi + (\nabla \xi)^T) \cdot Y.
\ee

We begin by showing that the ansatz \eqref{gBqs} is automatically Euclidean divergence-free, has flux surfaces and is quasisymmetric until a further conditions.
\begin{prop}
  [Characterization of quasisymmetric $B_g$ MHS solutions] \label{propqsss}
Let $\xi$ be a non-vanishing and divergence-free vector field tangent to $\partial T$ and let $g$ be any metric with $\L_\xi g = 0$.
 Let $\psi:T\to \mathbb{R}$ satisfy $\L_\xi \psi=0$ and $|\nabla \psi|>0$.
Then $B_g$ given by \eqref{gBqs} is (Euclidean) divergence-free, satisfies
\eqref{qs2intro} and is tangent to $\partial T$.
Moreover $B_g$ is weakly quasisymmetric if and only  if
\be\label{magBcond}
 (\L_\xi \delta) (\xi ,\xi)  +2 C^{-1}(\psi)  (\L_\xi \delta) (\xi ,\nabla_g^\perp \psi ) +C^{-2}(\psi) (\L_\xi \delta) (\nabla_g^\perp \psi ,\nabla_g^\perp \psi )=0.
\ee
The field $B_g$ additionally solves MHS with forcing $f$ if and only if
$
 f \cdot_g \nabla_g^\perp \psi = f \cdot_g \xi = 0,$
and $\psi$ satisfies the generalized Grad-Shafranov equation
\begin{equation}
  \div_g\bigg(\sqrt{|g|} \frac{\nabla_g \psi}{|\xi|_g^2} \bigg) - C(\psi)
 \frac{\xi}{|\xi|_g^2} \cdot_g \curl_g
 \left(\frac{\xi}{|\xi|_g^2}\right)
+ \frac{C(\psi)C'(\psi)}{\sqrt{|g|}|\xi|_g^2}+
\frac{P'(\psi)}{\sqrt{|g|}}
= \frac{f\cdot_g \nabla_g \psi}{\sqrt{|g|}|\nabla_g \psi|_g^2}.
\label{gengsapp}
\end{equation}

\end{prop}

This section will build up to the proof of  Proposition \ref{propqsss} by developing the following Lemmas \ref{propBlem}--\ref{curlBBlem}.  The proof is a straightforward combination of these results.
First we record some elementary vector identities.
\begin{lemma}\label{baselem}
  Fix a metric $g$. Let $\xi$ be a vector field with $|\xi| \neq 0$
  and let $\psi: T\to \mathbb{R}$ be a function satisfying $\L_\xi \psi = 0$. Then we have
 \be
 \xi \times_g \nabla_g \psi = \nabla_g^\perp \psi, \qquad \nabla_g\psi \times_g \nabla_g^\perp \psi =|\nabla_g\psi|_g^2 \xi , \qquad \nabla_g^\perp \psi \times_g \xi =  |\xi|_g^2 \nabla_g\psi.
 \label{basisidentities}
 \ee
 where we have introduced $\nabla_g^\perp \psi= \xi \times_g \nabla_g \psi$.
  Thus, the triple $( \nabla_g \psi, \nabla^\perp_g \psi, \xi)$ forms an orthogonal basis of $\mathbb{R}^3$  at each $x\in T$ where $|\nabla_g \psi|_g>0$.
\end{lemma}
\begin{proof}
Follows from the identity \eqref{tripleproduct}.
\end{proof}

The following are the main results in this section and are proved at
the end of the section.
\begin{lemma}[Structural properties of $B_g$] \label{propBlem1}
  Fix a metric $g$. Let $\xi: T\to \mathbb{R}^3$ be a (Euclidean) divergence-free vector field with $|\xi|_g \neq 0$ which is tangent to $\partial T$.
Let $\psi: T\to \mathbb{R}$ be a function satisfying $\L_\xi \psi = 0$ which is constant on $\partial T$.
   Fix  $C:\mathbb{R}\to \mathbb{R}$.
   Then $B_g$ defined in \eqref{gBqs} satisfies
   \begin{align}
   \xi\times {B}_g   &=-\nabla \psi, \label{fluxbg1}
   \\
    \div B_g &=   -\frac{1}{|\xi|_g^2} (\L_\xi g)(\xi, B_g),\label{divergencefree}
    \\ \label{symm}
        \L_\xi B_g &=  -\frac{1}{|\xi|_g^2} (\L_\xi g)(\xi, B_g) \xi\\
        \label{tang}
       B_g\cdot \hat{n}|_{\partial T}  &= 0.
   \end{align}
\end{lemma}
For the proof, see Section \ref{propBlem1pfsec}. The crucial point here
is despite the fact that $B_g$ is defined in terms of an arbitrary metric
$g$, the identities \eqref{fluxbg1} and \eqref{divergencefree} involve the
Euclidean metric.

We now begin the derivation of the Grad-Shafranov equation \eqref{gGS} which
involves a somewhat lengthy calculation using the above identities. The most
important and complicated ingredient is the following formula, which is
a direct consequence of Lemma \ref{auxlem} below.
\begin{lemma}[Curl of $B_g$]  \label{propBlem}
  Fix a metric $g$. Let $\xi$ be a vector field with $|\xi| \neq 0$
  and let $\psi: T\to \mathbb{R}$ be a function satisfying $\L_\xi \psi = 0$.
   Fix a function $C:\mathbb{R}\to \mathbb{R}$.
   Then $B_g$ defined in \eqref{gBqs} satisfies
   \begin{align}
    \curl_g B_g &= F \nabla_g \psi + G \nabla^\perp_g \psi + H \xi,
    \label{curlID}
   \end{align}
  where $\curl_g$ is with respect to the
   metric $g$, defined in \eqref{curlgdef}, and
with $F, G, H$ defined by
   \begin{align}
F &:=-\frac{\sqrt{|g|}}{|\xi|_g^4 |\nabla_g \psi|_g^2}
\bigg[\frac{C(\psi)}{\sqrt{|g|}}(\L_\xi g)(\xi, \nabla_g^\perp \psi) + 2 |\xi|_g^2 |\nabla_g \psi|_g^2 \frac{\L_\xi \sqrt{|g|}}{\sqrt{|g|}} \\
&\qquad\qquad\qquad\qquad  - |\xi|_g^2 (\L_\xi g)(\nabla_g \psi, \nabla_g \psi)
 - |\nabla_g \psi|_g^2 (\L_\xi g)(\xi,\xi) \bigg],\label{Fform}
\\ \label{Gform}
G
&:=
\frac{\sqrt{|g|}}{|\xi|_g^2} \frac{ 1}{|\nabla_g \psi|_g^2} \bigg[  (\L_\xi g)(B_g, \nabla_g \psi)
- |\nabla_g\psi|_g^2 \frac{C'(\psi)}{\sqrt{|g|}} \bigg] ,
\\ \label{Hform}
H&:= \div_g\bigg(\sqrt{|g|} \frac{\nabla_g \psi}{|\xi|_g^2} \bigg) + \frac{1}{|\xi|^4_g}
C(\psi)
\xi \cdot_g \curl_g \xi + \frac{\sqrt{|g|}}{|\xi|_g^2} (\L_\xi g)(\nabla_g \psi, \xi).
   \end{align}
\end{lemma}

\begin{lemma}[MHS for $B_g$]\label{curlBBlem}
  Fix a metric $g$. Let $\xi$ be a vector field with $|\xi| \neq 0$
  and let $\psi: T\to \mathbb{R}$ be a function satisfying $\L_\xi \psi = 0$.
  Fix a function $C:\mathbb{R}\to \mathbb{R}$.       Then $B_g$ defined in \eqref{gBqs} satisfies
 \begin{align}
   \curl_g B_g \times_g B_g - \nabla_g P
   &=
  \left(C(\psi) G -H- P'\right) \nabla_g \psi
   -  \frac{C(\psi)}{|\xi|_g^2}F \nabla_g^\perp \psi
   +  \frac{|\nabla_g\psi|_g^2}{|\xi|_g^2 } F\xi,
   \label{JtimesB}
 \end{align}
   with $F$, $G$ and $H$ defined by \eqref{Fform}, \eqref{Gform} and  \eqref{Hform}.
In particular, if $\L_\xi g = 0$ then $B$ satisfies the MHS equation with
force $f$
\begin{equation}
 \curl_g B_g \times_g B_g = \nabla_g P + f,
 \label{}
\end{equation}
if and only if
$
 f \cdot_g \nabla_g^\perp \psi = f \cdot_g \xi = 0,$
and $\psi$ satisfies the generalized Grad-Shafranov equation
\begin{equation}
  \div_g\bigg(\sqrt{|g|} \frac{\nabla_g \psi}{|\xi|_g^2} \bigg) - C(\psi)
 \frac{\xi}{|\xi|_g^2} \cdot_g \curl_g
 \left(\frac{\xi}{|\xi|_g^2}\right)
+ \frac{C(\psi)C'(\psi)}{\sqrt{|g|}|\xi|_g^2}+
\frac{P'(\psi)}{\sqrt{|g|}}
= \frac{f\cdot_g \nabla_g \psi}{\sqrt{|g|}|\nabla_g \psi|_g^2}.
 \label{gs}
\end{equation}

 \end{lemma}
\begin{proof}
Follows from Lemma \ref{propBlem}, \eqref{baselem},  standard vector identities and $\L_\xi \psi=0$.
\end{proof}

The generalized Grad--Shafranov equation \eqref{gs} for vector
fields of the form \eqref{gBqs} was first derived in
\cite{Burby} when $g$ was taken to be the circle-averaged metric.

\begin{lemma}[Quasisymmetry of $B_g$]\label{propBLem}
  Fix a metric $g$ with $\L_\xi g = 0$. Let $\xi$ be a vector field with $|\xi| \neq 0$
  and let $\psi: T\to \mathbb{R}$ be a function satisfying $\L_\xi \psi = 0$.
  Fix  $C:\mathbb{R}\to \mathbb{R}$.    Then $B_g$  satisfies
\be
\L_\xi |B_g|^2 = \frac{C^2(\psi)}{|\xi|_g^4} \bigg[ (\L_\xi \delta) (\xi ,\xi)  +2 C^{-1}(\psi)  (\L_\xi \delta) (\xi ,\nabla_g^\perp \psi ) +C^{-2}(\psi) (\L_\xi \delta) (\nabla_g^\perp \psi ,\nabla_g^\perp \psi ) \bigg].
\ee
\end{lemma}

\subsection{Auxiliary Lemmas}
We collect some calculations which are useful for the proofs of the other Lemmas in the following statement.
\begin{lemma}\label{auxlem}
  Fix a metric $g$. Let $\xi$ be a vector field with $|\xi| \neq 0$
  and let $\psi: T\to \mathbb{R}$ be a function satisfying $\L_\xi \psi = 0$.
  Fix a function $C:\mathbb{R}\to \mathbb{R}$.
  Then
\begin{align}\label{divform1}
\div_g \left(C(\psi) \frac{\xi}{|\xi|_g^2}  \right) &=  \frac{C(\psi)}{|\xi|_g^2} \bigg( \div_g \xi  -  \frac{1}{|\xi|_g^2}(\L_\xi g)(\xi, \xi)\bigg),\\ \label{divform2}
\div_g \left(\sqrt{|g|} \frac{\nabla^\perp_g\psi }{|\xi|_g^2}  \right) &= -\frac{\sqrt{|g|}}{|\xi|_g^4}(\L_\xi g)( \xi,\nabla_g^\perp \psi) +   \frac{1}{|\xi|_g^2} \L_{\nabla^\perp_g\psi } \sqrt{|g|},\\ \nonumber
\curl_g  \left(C(\psi) \frac{\xi}{|\xi|_g^2}  \right)  &=   \frac{1}{|\xi|_g^4} C(\psi)(\xi \cdot_g \curl_g \xi) \xi  + \frac{1}{|\xi|_g^2}\bigg(  \frac{C(\psi) }{|\xi|_g^2|\nabla_g\psi|_g^2}  (\L_\xi g)( \xi, \nabla_g \psi )- C'(\psi) \bigg) \nabla_g^\perp \psi \\\label{curlform1}
&\qquad\qquad   -   \frac{C(\psi) }{|\xi|_g^4|\nabla_g\psi|_g^2}    (\L_\xi g)( \xi, \nabla_g^\perp \psi ) \nabla_g \psi  \\
\curl_g \left(\sqrt{|g|} \frac{\nabla^\perp_g\psi }{|\xi|_g^2}  \right) &= \bigg(\div_g \left( \sqrt{|g|} \frac{\nabla_g \psi }{|\xi|_g^2} \right)+   \frac{\sqrt{|g|} }{|\xi|_g^4}(\L_\xi g)(\nabla_g \psi,\xi)  \bigg)\xi \\
&\qquad + \frac{\sqrt{|g|} }{|\xi|_g^4}  \Big(   \L_\xi g(\xi,\xi)- 2|\xi|_g^2 \frac{\L_\xi \sqrt{|g|}}{\sqrt{|g|}} +\frac{|\xi|_g^2}{|\nabla_g \psi|^2}(\L_\xi g)(\nabla_g \psi,\nabla_g \psi)\Big)  \nabla_g \psi\\
&\qquad\qquad
+  \frac{\sqrt{|g|}}{|\xi|_g^4|\nabla_g\psi|^2 }(\L_\xi g)(\nabla_g \psi,\nabla_g^\perp \psi)\nabla_g^\perp \psi. \label{curlform2}
\end{align}
\end{lemma}

\begin{proof}
We will repeatedly use the product rule \eqref{product1} as well
 as the commutator identity
\be\label{commu}
\L_\xi \nabla_g f = \nabla_g \L_\xi f - ( \L_\xi g)( \nabla_g f, \cdot).
\ee

\noindent \textbf{Step 1: Identity \eqref{divform1}.} To prove \eqref{divform1} we note
\begin{align}
\div_g \left(C(\psi) \frac{\xi}{|\xi|_g^2}  \right) &= C(\psi) \frac{\div_g \xi}{|\xi|_g^2}   - |\xi|_g^{-4}C(\psi) \L_\xi |\xi|_g^{2}= \frac{C(\psi)}{|\xi|_g^2} \div_g \xi   - \frac{C(\psi)}{|\xi|_g^4} (\L_\xi g )(\xi,\xi),
\end{align}
using the product rule \eqref{product}.\\

\noindent \textbf{Step 2: Identity \eqref{divform2}.}
First note that
\be
\div_g \left(\sqrt{|g|} \frac{\nabla^\perp_g\psi }{|\xi|_g^2}  \right)=\sqrt{|g|}  \div_g \left(\frac{\nabla^\perp_g\psi }{|\xi|_g^2}  \right) + \frac{1}{|\xi|_g^2} \L_{\nabla^\perp_g\psi } \sqrt{|g|}.
\ee
Next we compute
\begin{align}
\div_g \left(\frac{\nabla^\perp_g\psi }{|\xi|_g^2}  \right)&= \frac{1}{|\xi|_g^2}\bigg(\div_g \nabla^\perp_g\psi +|\xi|_g^2 \L_{ \nabla^\perp_g\psi} |\xi|_g^{-2} \bigg)\\
&= \frac{1}{|\xi|_g^2}\bigg(\div_g( \xi \times_g \nabla_g \psi) -  |\xi|_g^{-2}(\xi \times_g \nabla_g \psi) \cdot_g \nabla_g |\xi|_g^{2}\bigg)\\
&= \frac{1}{|\xi|_g^2}\bigg(\L_{\curl_g\! \xi} \psi  -  \xi \cdot_g \curl_g \nabla_g \psi- |\xi|_g^{-2} ( \nabla_g |\xi|_g^{2}  \times_g \xi) \cdot_g \nabla_g \psi \bigg)\\
&= \frac{1}{|\xi|_g^2}\bigg(\L_{\curl_g\! \xi} \psi - |\xi|_g^{-2} ( \nabla_g |\xi|_g^{2}  \times_g \xi) \cdot_g \nabla_g \psi \bigg). \label{formc}
\end{align}
We now simplify the second term in the above.  First note the identity (which follows from \eqref{lastid})
\begin{align} \nonumber
\xi \times_g \curl_g \xi &= \frac{1}{2} \nabla_g |\xi|_g^2 - (\xi \cdot_g \nabla_g) \xi\\
 &=  \nabla_g |\xi|_g^2 - ((\xi \cdot_g \nabla_g) \xi+ \nabla_g \xi \cdot_g \xi)\\
 & =   \nabla_g |\xi|_g^2 - (\L_\xi g)\cdot_g \xi.
 \label{magicidentity}
\end{align}
so that
\begin{align}
\nabla_g |\xi|_g^2 \times_g \xi &= (\xi \times_g \curl_g \xi )\times_g \xi + ((\L_\xi g)\cdot_g\xi)\times_g\xi\\
&= |\xi|_g^2 \curl_g \xi - (\xi \cdot_g \curl_g \xi) \xi + ((\L_\xi g)\cdot_g \xi)\times_g\xi,\label{crossform}
\end{align}
where we have used the elementary identity
 \be
  (\xi \times_g \curl_g \xi ) \times_g \xi =   |\xi|_g^2 \curl_g \xi - (\xi \cdot_g \curl_g \xi) \xi.
  \label{xitimescurl}
 \ee
 Noting finally that
 \be\label{gswitch}
  (((\L_\xi g)\cdot_g \xi)\times_g\xi )\cdot_g \nabla_g\psi= ((\L_\xi g)\cdot_g \xi)\cdot_g (\xi\times_g \nabla_g\psi)=(\L_\xi g)( \xi,\nabla_g^\perp \psi),
 \ee
using that  $\L_\xi \psi=0$ we have
\begin{align}
 |\xi|_g^{-2} ( \nabla_g |\xi|_g^{2}  \times_g \xi)  \cdot_g \nabla_g \psi   &=\L_{\curl_g\! \xi} \psi + |\xi|_g^{-2} (\L_\xi g)( \xi,\nabla_g^\perp \psi).
\end{align}
Putting this together with \eqref{formc}, we obtain the identity \eqref{divform2}.
\\

\noindent \textbf{Step 3: Identity \eqref{curlform1}.}
To prove \eqref{curlform1}, we note
\begin{equation}
 \curl_g (C(\psi) \xi) = C'(\psi)\nabla_g \psi \times_g \xi + C(\psi) \curl_g \xi
 = -C'(\psi) \nabla_g^\perp \psi + C(\psi) \curl_g \xi.
\end{equation}
Using this formula and \eqref{magicidentity}, we find that
\begin{align}
\curl_g  \left(C(\psi) \frac{\xi}{|\xi|_g^2}  \right) &= \frac{1}{|\xi|_g^2}\bigg( -C'(\psi) \nabla_g^\perp \psi + C(\psi) \curl_g \xi\bigg) -  |\xi|_g^{-4}  C(\psi) \nabla |\xi|_g^2 \times_g \xi\\
&= \frac{1}{|\xi|_g^2}\bigg( -C'(\psi) \nabla_g^\perp \psi + C(\psi) \curl_g \xi\bigg)\\
&\qquad -  |\xi|_g^{-4}  C(\psi)\bigg(  |\xi|_g^2 \curl_g \xi - (\xi \cdot_g \curl_g \xi) \xi+  (\L_\xi g)\cdot \xi \times_g \xi\bigg)
\\
&= -\frac{C'(\psi)}{|\xi|_g^2} \nabla_g^\perp \psi + |\xi|_g^{-4}  C(\psi)\bigg(  (\xi \cdot_g \curl_g \xi) \xi-  (\L_\xi g)\cdot \xi \times_g \xi\bigg).
\label{curlgCpsi}
\end{align}
Note finally using  Lemma \ref{baselem} that
\begin{align}
(\L_\xi g)\cdot_g \xi \times_g \xi&=((\L_\xi g)\cdot_g \xi \times_g \xi) \cdot_g\widehat{ \nabla_g \psi}\ \widehat{\nabla_g \psi }+ ((\L_\xi g)\cdot_g \xi \times_g \xi) \cdot_g \widehat{\nabla_g^\perp \psi} \ \widehat{\nabla_g^\perp \psi }\\
&= \frac{1}{|\nabla_g\psi|_g^2}  (\L_\xi g)( \xi, \nabla_g^\perp \psi ) \nabla_g \psi - \frac{1}{|\nabla_g\psi|_g^2}  (\L_\xi g)( \xi, \nabla_g \psi ) \nabla_g^\perp \psi,
\end{align}
where we used the identity \eqref{gswitch} in passing to the second line together with
 \be\label{gswitch2}
  (((\L_\xi g)\cdot_g \xi)\times_g\xi )\cdot_g \nabla_g^\perp\psi= ((\L_\xi g)\cdot_g \xi)\cdot_g (\xi\times_g  \nabla_g^\perp\psi)=- |\xi|_g^2 (\L_\xi g)( \xi,\nabla \psi).
 \ee
Combining this with \eqref{curlgCpsi} gives
\begin{align}
\curl_g  \left(C(\psi) \frac{\xi}{|\xi|_g^2}  \right)
&=  -\frac{C'(\psi)}{|\xi|_g^2} \nabla_g^\perp \psi + |\xi|_g^{-4}  C(\psi)  (\xi \cdot_g \curl_g \xi) \xi \\
&\qquad+ \frac{1}{|\xi|_g^2}   \bigg(  \frac{C(\psi)}{|\xi|_g^2|\nabla_g\psi|_g^2}  (\L_\xi g)( \xi, \nabla_g \psi ) \nabla_g^\perp \psi - \frac{C(\psi)}{|\xi|_g^2|\nabla_g\psi|_g^2}  (\L_\xi g)( \xi, \nabla_g^\perp \psi ) \nabla_g \psi \bigg).
\end{align}
Rearrangement establishes \eqref{curlform1}.
\\

\noindent \textbf{Step 4: Identity \eqref{curlform2}.} First note that
\begin{align}
\curl_g \left(\sqrt{|g|} \frac{\nabla^\perp_g\psi }{|\xi|_g^2}  \right) &= \sqrt{|g|} \curl_g \left(\frac{\nabla^\perp_g\psi }{|\xi|_g^2}  \right) + \frac{1}{|\xi|_g^2} \nabla  \sqrt{|g|} \times_g \nabla^\perp_g\psi\\
&= \sqrt{|g|} \curl_g \left(\frac{\nabla^\perp_g\psi }{|\xi|_g^2}  \right) + \frac{1}{|\xi|_g^2}  (\L_{\nabla_g \psi} \sqrt{|g|}) \xi -  \frac{1}{|\xi|_g^2}  (\L_{\xi} \sqrt{|g|}) \nabla \psi.
\end{align}
 Now, by the identity \eqref{magic1},
\begin{align}
  \curl_g \nabla^\perp_g \psi =   \curl_g (\xi\times_g \nabla_g \psi) &=
  \xi \Delta_g \psi - \nabla_g \psi \div_g \xi
  + \L_{\nabla_g \psi}\xi\\
  &
  =
  \xi \Delta_g \psi - \nabla_g \psi \div_g \xi
  - \L_{\xi}\nabla_g \psi\\
    &=
  \xi \Delta_g \psi - \nabla_g \psi \div_g \xi
  + (\L_\xi g)(\nabla_g \psi,\cdot),
  \label{identity1}
\end{align}
where we used \eqref{commu} and $(\L_\xi g)(\nabla_g \psi,\cdot)$ is defined as in \eqref{abuse}.
Therefore
\begin{align}
\sqrt{|g|} \curl_g \left(\frac{\nabla^\perp_g\psi }{|\xi|_g^2}  \right) &=   \frac{\sqrt{|g|} }{|\xi|_g^2}\Big( \xi \Delta_g \psi - \div_g \xi\  \nabla_g \psi
  +  (\L_\xi g)(\nabla_g \psi,\cdot)\Big) -  \frac{\sqrt{|g|} }{|\xi|_g^4} \nabla_g |\xi|_g^2 \times \nabla^\perp_g\psi \\
  &=   \frac{\sqrt{|g|} }{|\xi|_g^2}\Big( \xi \Delta_g \psi - \div_g \xi\  \nabla_g \psi
 +  (\L_\xi g)(\nabla_g \psi,\cdot)\Big) \\
  &\qquad +  \frac{\sqrt{|g|} }{|\xi|_g^4}\bigg( ( \L_\xi  |\xi|_g^2)   \nabla_g\psi - (  \nabla_g\psi\cdot_g \nabla_g |\xi|_g^2 ) \xi \bigg)\\
    &=   \sqrt{|g|} \div_g \left( \frac{\nabla_g \psi }{|\xi|_g^2} \right)  \xi
 +  \frac{\sqrt{|g|} }{|\xi|_g^2}(\L_\xi g)(\nabla_g \psi,\cdot) + \frac{\sqrt{|g|} }{|\xi|_g^4}  \Big( \L_\xi g(\xi,\xi)-  |\xi|_g^2 \div_g \xi  \Big)  \nabla_g \psi\\
     &=   \div_g \left( \sqrt{|g|} \frac{\nabla_g \psi }{|\xi|_g^2} \right)  \xi - \frac{1}{|\xi|_g^2} (\L_{\nabla_g \psi} \sqrt{|g|} ) \xi\\
     &\qquad
 +  \frac{\sqrt{|g|} }{|\xi|_g^2}(\L_\xi g)(\nabla_g \psi,\cdot) + \frac{1}{|\xi|_g^4}  \Big(\sqrt{|g|}  \L_\xi g(\xi,\xi)-  |\xi|_g^2 \L_{\xi}  \sqrt{|g|}  \Big)  \nabla_g \psi,
\end{align}
where we \eqref{divconv} to say $  \sqrt{|g|} \div_g\xi = \L_{\xi}  \sqrt{|g|}$ as well as  the identity
\be
 \nabla_g |\xi|_g^2 \times \nabla^\perp_g\psi :=  \nabla_g |\xi|_g^2 \times (\xi \times_g \nabla_g \psi)= (\nabla_g |\xi|_g^2 \cdot_g \nabla_g \psi) \xi -  (\nabla_g |\xi|_g^2 \cdot_g \xi) \nabla_g \psi.
\ee
Finally, note that we can express
\begin{align}
(\L_\xi g)(\nabla_g \psi,\cdot) &= \frac{1}{|\xi|_g^2} (\L_\xi g)(\nabla_g \psi,\xi ) \xi +  \frac{1}{|\xi|_g^2|\nabla_g \psi|_g^2 } (\L_\xi g)(\nabla_g \psi,\nabla^\perp_g\psi) \nabla^\perp_g\psi\\
&\qquad +  \frac{1}{|\nabla_g \psi|_g^2 } (\L_\xi g)(\nabla_g \psi,\nabla_g \psi) \nabla_g\psi.
\end{align}
This completes the derivation.
\end{proof}

\subsection{Proof of Lemma \ref{propBlem1}}\label{propBlem1pfsec}  The result follows from direct computation
as follows.
\\

\noindent \textbf{Step 1: Identity \eqref{fluxbg1}.}
The property of having a flux function \eqref{fluxbg1} follows from
Lemma \ref{timestimesglem}.
\\

\noindent \textbf{Step 2: Identity \eqref{divergencefree}.}
For the divergence \eqref{divergencefree}, Lemma \ref{propBLem} gives
\be\label{gdivB}
    \div_g B_g =  \frac{1}{|\xi|_g^2} \bigg( C(\psi)\div_g \xi  - (\L_\xi g)(\xi, B_g)\bigg)  +   \frac{1}{|\xi|_g^2} \L_{\nabla^\perp_g\psi } \sqrt{|g|}.
    \ee
Next recall the relation between the divergence on flat and curved backgrounds
 \begin{equation}
  \div X = \div_g X - \frac{1}{\sqrt{|g|}} \L_X\sqrt{|g|}.
  \label{divconv}
 \end{equation}
Applying this identity to convert \eqref{gdivB} to the divergence using the Euclidean metric, we have
 \begin{align}
    \div B_g  &=     \div_g B_g- \frac{1}{\sqrt{|g|}} \L_{B_g}\sqrt{|g|}  = \frac{1}{|\xi|_g^2} \big(C(\psi) \div_g\xi - (\L_\xi g)(\xi, B)\big) - \frac{1}{\sqrt{|g|}}  \frac{C(\psi) }{|\xi|_g^2}  \L_{\xi}  \sqrt{|g|}.
    \end{align}
    Using  $\div \xi=0$ and \eqref{divconv} again we find $    \sqrt{|g|} \div_g\xi = \L_{\xi}  \sqrt{|g|},$
and get the claimed result.
\\

\noindent \textbf{Step 3: Identity \eqref{symm}.}
We have the identity
\begin{align}
\L_\xi B_g&:= \xi \cdot \nabla B_g - B_g\cdot \nabla \xi  \\
&= \curl(B_g \times \xi) +  (\div B_g) \xi  - (\div \xi) B_g= -\frac{1}{|\xi|_g^2} (\L_\xi g)(\xi, B_g) \xi,
\end{align}
and the result follows from \eqref{fluxbg1}, \eqref{divergencefree} and the
assumption $\div \xi = 0$.
\\

\noindent \textbf{Step 4: Identity \eqref{tang}.}
Let $\hat{n}$ be the unit outward normal vector to $\partial T$.  Then we have
\begin{equation}
 B_g\cdot  \hat{n} = \frac{1}{|\xi|_g^2} \sqrt{|g|} \  (\xi \times_g \nabla_g \psi)\cdot  \hat{n},
\end{equation}
since $  \xi \cdot  \hat{n}=0$ by assumption. Now, for any vector field $X$ and scalar function $f$ we have
\be
X\cdot \nabla f  = \delta_{ij} X^i \delta^{jk} \partial_k f = \delta_i^k X^i \partial_k f = g_{im} g^{km} X^i \partial_k f = g_{im} X^i (\nabla_g f)^m = X \cdot_g \nabla_g f.
\ee
As a result,  since $\psi$ is assumed constant on the boundary, we can choose $\hat{n} = \nabla \psi/|\nabla \psi|$ on the boundary and a standard vector identity shows that $ (\xi \times_g \nabla_g \psi)\cdot  \hat{n} =0.$

\subsection{Proof of Lemma \ref{propBLem}}

\begin{proof}
Direct computation shows
\be
|B_g|^2 = \frac{1}{|\xi|_g^4} \bigg[ C(\psi) |\xi|^2 +2C(\psi) \xi \cdot \nabla_g^\perp \psi + |\nabla_g^\perp \psi|^2 \bigg].
\ee
Since $\L_\xi g = 0$, from \eqref{symm} it follows that
 $\L_\xi B_g=0$.  Thus
we have
\be
\L_{\xi} \nabla_g^\perp \psi = \L_{\xi} \big(|\xi|_g^2 B_g-  C(\psi) \xi\big) = 0.
\ee
Using $\L_\xi|\xi|_g^2=0$, $\L_\xi \xi=0$, $\L_\xi \psi=0$, $\L_{\xi} \nabla_g^\perp \psi=0$ and  $\L_\xi |\xi|^2 = (\L_\xi \delta) (\xi ,\xi) $ completes the proof.
\end{proof}

\section{Explicit expression for the generalized Grad-Shafranov equation}\label{explicitGS}

Fix a domain $D$ in the $\{\Phi = 0\}$ half-plane and let $\xi$ be a vector
field whose orbits starting from $D$ are all periodic (with possibly different
period). Fix an arbitrary local coordinate system on $D$ and extend it to a
coordinate system $(x_1, x_2, x_3)$ on the
torus $T$ defined in \eqref{torusdef} by pulling back along the flow of $\xi$.
In these coordinates we have $\xi\cdot \nabla f = \frac{\pa}{\pa x^3} f$.
In this section we express the coefficients appearing in the
generalized Grad-Shafranov equation \eqref{gGS} in these coordinates.
The most complicated part of the calculation is contained in
the following lemma.

\begin{lemma}
  Let $g$ be an arbitrary metric on $T$ and let $(x_1, x_2, x_3)$
  be a coordinate system on $T$ as above. Then
  \begin{multline}
   \curl_g \xi \cdot_g \xi
   =
   |g|^{1/2} \big( (\pa_1 g_{23} - \pa_2 g_{13} ) (g^{11} g^{22} - (g^{12})^2)\\
   +
 (\pa_3 g_{13} - \pa_1 g_{33}) ( g^{21} g^{23} - g^{22} g^{13})
   + (\pa_3 g_{23} - \pa_2 g_{33}) ( g^{11} g^{23} - g^{12} g^{13})\big)
 \label{curlcdotidentityexplicit}
  \end{multline}
\end{lemma}

\begin{proof}
We use the formula
 $\curl_g \xi \cdot_g \xi = i_\xi (\curl_g \xi)_{\flat} = i_\xi
 (*_g d \alpha),$
where $*_g$ is the Hodge star in terms of $g$ and $\alpha = \xi_\flat$ denotes
the one-form which is dual to $\xi$ with respect to $g$. Explicitly
 $\alpha = \alpha_i dx^i = g_{ij}\xi^j dx^i = g_{i3} dx^i$.
We now compute the terms on the right-hand side of \eqref{curlcdotidentityexplicit} explicity and the main step
is computing $*_g d\alpha$.
Acting on two-forms, $*_g$ is defined by linearity and the rule
\begin{equation}
 *_g (dx^i\wedge d x^j) =
 |g|^{1/2} g^{ik} g^{j\ell} \epsilon_{k\ell m} d x^m,
 \label{}
\end{equation}
where $|g| = \det g$ and
$\epsilon_{k\ell m}$ is the Levi-Civita symbol.

Since in our coordinate system $\xi = \pa_{3}$ we have
 $i_\xi d x^m = d x^m(\pa_{3}) = \delta^{m3}$ and so
\begin{equation}
 i_\xi *_g (dx^i \wedge \rmd x^j)
 =
 |g|^{1/2} g^{ik} g^{j\ell} \epsilon_{k\ell 3}.
 \label{}
\end{equation}
A straightforward calculation shows
\begin{align}
 i_{\xi} *_g (dx^1\wedge  d x^2)
 = |g|^{1/2} \big( g^{11} g^{22} - (g^{12})^2\big),\\
 i_\xi *_g (dx^2 \wedge d x^3)
 = |g|^{1/2} \big( g^{21} g^{32} - g^{22}g^{31}\big),\\
 i_\xi *_g (dx^1 \wedge d x^3)
 = |g|^{1/2} \big( g^{11} g^{32} - g^{12} g^{31}\big).
 \label{}
\end{align}
Since
 $d \alpha = (\pa_1 \alpha_2 - \pa_2 \alpha_1) dx^1\wedge dx^2
 + (\pa_1 \alpha_3 - \pa_3 \alpha_1) dx^1 \wedge dx^3
 + (\pa_2 \alpha_3 - \pa_3 \alpha_2) dx^2\wedge dx^3,
 $ we have
\begin{align} \nonumber
  \curl_g \xi  \cdot_g \xi &= (\pa_1 \alpha_2 - \pa_2 \alpha_1) i_\xi *_g (dx^1\wedge  d x^2)
 - \pa_1 \alpha_3 i_\xi *_g (dx^1 \wedge d x^3) \nonumber
 + \pa_2 \alpha_3 i_\xi *_g (dx^2\wedge d x^3)\\ \nonumber
 &= |g|^{1/2} \big( (\pa_1\alpha_2 - \pa_2 \alpha_1)  (g^{11} g^{22} - (g^{12})^2)
 +(\pa_3 \alpha_1 - \pa_1 \alpha_3) ( g^{21} g^{23} - g^{22} g^{13})\\\nonumber
 &\qquad\qquad+ (\pa_3 \alpha_2 - \pa_2 \alpha_3) ( g^{11} g^{23} - g^{12} g^{13})\big)
\end{align}
which gives \eqref{curlcdotidentityexplicit} since $\alpha_i = g_{i3}$.
\end{proof}

The next lemma follows from the previous one and \eqref{gs}
after noting that $|\xi|^2 = g(\xi, \xi) = g_{33}$.
\begin{lemma}
  Fix a vector field $\xi$ and a metric $g$ with $\L_\xi g = 0$.
  Let $(x_1, x_2, x_3)$ be any
   coordinate system as in the statement of the previous Lemma.
   Then equation \eqref{gs} with $f = 0$ takes the form
   \begin{equation}
    \sum_{i,j=1}^{3}
    a^{ij}_{\xi,g} \pa_i \pa_j \psi
    + \sum_{i = 1}^3 b_{\xi,g}^i \pa_i \psi
    + G_{\xi,g}(x_1, x_2, x_3, C, \psi) + \frac{P'(\psi)}{\sqrt{|g|}} = 0,
    \label{}
   \end{equation}
   where
   \begin{align}
    a^{ij}_{\xi,g} &= \frac{\sqrt{|g|}}{g_{33}}
    g^{ij},\quad
    b^i_{\xi,g} =
     \sum_{j =1,2} \frac{\sqrt{|g|}}{g_{33}}\pa_j\big( \sqrt{|g|} g^{ij}\big)
    + g^{ij}\pa_j \bigg( \frac{\sqrt{|g|}}{g_{33}}\bigg),\quad   G_{\xi,g} &=\frac{C(\psi)}{ g_{33}}\left(\frac{C'(\psi)}{ \sqrt{|g|}}
    -
    \xi\cdot_g \curl_g \xi\right),
    \label{}
   \end{align}
   where $\xi\cdot_g\curl_g \xi$ is given by
   \eqref{curlcdotidentityexplicit}.
\end{lemma}

 \subsection*{Acknowledgments}   We thank A. Bhattacharjee, J. Burby,  A. Cerfon,  N. Kallinikos, M. Landreman, R. MacKay and G. Misio\l{}ek for  insightful discussions.   The work of PC was partially supported by NSF grant DMS-1713985 and by the
Simons Center for Hidden Symmetries and Fusion Energy  award \# 601960.
Research of TD was partially supported by
NSF grant DMS-1703997.  Research of DG was partially supported by
the Simons Center for Hidden Symmetries and Fusion Energy.

\end{document}